\documentclass[11pt]{article}

\usepackage[ansinew]{inputenc}
\usepackage{amsmath,amssymb}
\usepackage{color}
\usepackage{a4wide}
\usepackage{verbatim}
\usepackage{scrtime}
\usepackage{float}
\usepackage{graphicx}
\usepackage{graphics}
\usepackage{pstricks}
\usepackage{multirow}
\usepackage[ruled,vlined]{algorithm2e}

\large\normalsize%

\def \rr {\mathbb{R}}

\newenvironment{proof}{{\it Proof.}}{\hspace{\stretch{1}} $\square$}

\newcommand{\pacomm}[1]{}

\newcommand{\cbcomm}[1]{}

\newcommand{\cdel}[1]{}

\newcommand{\mtns}[1]{}

\newtheorem{thrm}{Theorem}%[section]
\newtheorem{lmm}[thrm]{Lemma}
\newtheorem{crllr}[thrm]{Corollary}

\newtheorem{ssmptn}{Assumption}

\begin{document}

\title{Low-rank optimization for semidefinite convex problems\footnotemark[1]}
\author{M. Journée\footnotemark[2] \and F. Bach\footnotemark[3] \and P.-A. Absil\footnotemark[4] \and R. Sepulchre\footnotemark[2]}
\date{Compiled on \today, \thistime}

\renewcommand{\thefootnote}{\fnsymbol{footnote}}
 \footnotetext[2]{Department of Electrical
Engineering and Computer Science, University of Li\`{e}ge, 4000
Liège, Belgium. Email: [M.Journee, R.Sepulchre]@ulg.ac.be}
\footnotetext[3]{INRIA - Willow project, Département
d'Informatique, Ecole Normale Supérieure, 45, rue d'Ulm, 75230
Paris, France. Email: Francis.Bach@mines.org}
\footnotetext[4]{Department of Mathematical Engineering,
  Universit\'e catholique de Louvain, 1348 Louvain-la-Neuve,
  Belgium. URL: http://www.inma.ucl.ac.be/$\sim$absil/}
\footnotetext[1]{Michel Journ\'{e}e is a research fellow of the
Belgian National Fund for Scientific Research (FNRS). This paper
presents research results of the Belgian Network DYSCO (Dynamical
Systems, Control, and Optimization), funded by the Interuniversity
Attraction Poles Programme, initiated by the Belgian State,
Science Policy Office. The scientific responsibility rests with
its authors.}
\renewcommand{\thefootnote}{\arabic{footnote}}

\maketitle
\section*{Abstract}
We propose an algorithm for solving nonlinear convex programs
defined in terms of a symmetric positive semidefinite matrix
variable $X$. This algorithm rests on the factorization $X=Y Y^T$,
where the number of columns of $Y$ fixes the rank of $X$. It is
thus very effective for solving programs that have a low rank
solution. The factorization $X=Y Y^T$ evokes a reformulation of
the original problem as an optimization on a particular quotient
manifold. The present paper discusses the geometry of that
manifold and derives a second order optimization method. It
furthermore provides some conditions on the rank of the
factorization to ensure equivalence with the original problem. The
efficiency of the proposed algorithm is illustrated on two
applications: the maximal cut of a graph and the sparse
principal component analysis problem.\\

\section{Introduction}
\label{sec:intro} Many combinatorial optimization problems can be
relaxed into a convex program. These relaxations are mainly
introduced as a tool to obtain lower and upper bounds on the
problem of interest. The relaxed solutions provide approximate
solutions to the original program. Even when the relaxation is
convex, computing its solution might be a demanding task in the
case of large-scale problems. In fact, most convex relaxations of
combinatorial problems consist in expanding the dimension of the
search space by optimizing over a symmetric positive semidefinite
matrix variable of the size of the original problem. Fortunately,
in many cases, the relaxation is tight once its solution is rank
one, and it is expected that the convex relaxation, defined in
terms of a matrix variable that is likely to be very large,
presents a low-rank solution. This property can be exploited to
make a direct solution of the convex
problem feasible in large-scale problems.\\

The present paper focuses on the convex optimization problem,
\begin{equation}
\label{eq:1}
\begin{array}{ll}
\underset{X \in {\mathbb{S}}^{n}}{\min} & f(X)\\
\text{s.t.} & \mathrm{Tr}(A_i X)=b_i, \; A_i \in {\mathbb{S}}^{n}, b_i \in \mathbb{R}, \; i=1,\ldots,m,\\
& X \succeq 0, \\
\end{array}
\end{equation}
where the function $f$ is convex and ${\mathbb{S}}^{n}=\{X \in
\mathbb{R}^{n \times n} | X^T=X\}$ denotes the set of the
symmetric matrices of $\mathbb{R}^{n \times n}$. In general, the
solution of this convex program has to be searched in a space of
dimension $\frac{n (n+1)}{2}$. An approach is proposed for solving
(\ref{eq:1}) that is able to deal with a large dimension $n$ once
the following assumptions hold.
\begin{ssmptn}
The program (\ref{eq:1}) presents a low-rank solution $X^*$, i.e.,
\[\mathrm{rank}(X^*)=r \ll n.\]
\end{ssmptn}
\begin{ssmptn}
The symmetric matrices $A_i$ satisfy
\[ A_i A_j =0,\]
for any $i,j \in \{1,\ldots,m\}$ such that $i \neq j$.
\end{ssmptn}
Assumption 2 is fulfilled, e.g., by the spectahedron,
\[\mathcal{S}=\{ X \in {\mathbb{S}}^n| X \succeq 0, \mathrm{Tr}(X)=1\},\] and
the elliptope,\footnote{The elliptope is also known as the set of
correlation matrices.}
\begin{equation}
\label{eq:elliptope}
\mathcal{E}=\{ X \in {\mathbb{S}}^n|X\succeq
0,\mathrm{diag}(X)=\mathbf{1}\}.\end{equation} Assumption 1
suggests to factorize the optimization variable $X$ as
\begin{equation} \label{eq:fact} X=Y Y^T\end{equation} with $Y \in
\mathbb{R}^{n \times p}$. This leads to a nonlinear optimization
program in terms of the matrix $Y$,
\begin{equation}
\label{eq:2}
\begin{array}{ll}
\underset{Y \in \mathbb{R}^{n \times p}}{\min} & f(YY^T)\\
\text{s.t.} & \mathrm{Tr}(Y^T A_i Y)=b_i, \; A_i \in {\mathbb{S}}^{n}, b_i \in \mathbb{R},\; i=1,\ldots,m.\\
\end{array}
\end{equation}
Program (\ref{eq:2}) searches a space of dimension $n p$, which is
much lower than the dimension of the symmetric positive
semidefinite matrices $X$. However, this program is no longer
convex. \\

A further potential difficulty of the program (\ref{eq:2}) is that
the solutions are not isolated. For any solution $\tilde{Y}$ and
any orthogonal matrix $Q$ of $\mathbb{R}^{p \times p}$, i.e., $Q^T
Q = I$, the matrix $\tilde{Y} Q$ is also a solution. In other
words, the program (\ref{eq:2}) is invariant by right
multiplication of the unknown with an orthogonal matrix. This
issue is not harmful for simple gradient schemes but it greatly
affects the convergence of second order methods (see e.g.,
\cite{AbsMahSep2008} and \cite{Absil08}). In order to take into
account the inherent symmetry of the solution, the algorithm
developed in this paper does not optimize on the Euclidean space
$\mathbb{R}^{n \times p}$. Instead, one considers a search space,
whose points are the equivalence classes $\{ Y Q | Q \in
\mathbb{R}^{p \times p}, Q^T Q = I \}$. The minimizers of
(\ref{eq:2}) are isolated in
that \emph{quotient} space.\\

It is important to mention that the rank $r$ of the solution $X^*$
is usually unknown. The algorithm we propose for solving
(\ref{eq:1}) thus provides a method that finds a local minimizer
$Y_*$ of (\ref{eq:2}) with an approach that increments $p$ until a
sufficient condition is satisfied for $Y_*$ to provide the
solution $Y_*Y_*^T$ of (1). The proposed algorithm converges
monotonically towards the solution of (\ref{eq:1}), is based on
superlinear second order methods, and is provided with an
indicator of convergence able to control the accuracy of the
results.\\

The idea of reformulating a convex program into a nonconvex one by
factorization of the matrix unknown is not new and was
investigated in \cite{Burer03} for solving semidefinite programs
(SDP). While the setup considered in \cite{Burer03} is general but
restricted to gradient methods, the present paper further exploits
the particular structure of the equality constraints (Assumption
2) and proposes second-order methods that lead to a descent
algorithm with guaranteed superlinear convergence. The authors of
\cite{Grubisic07} also exploit the factorization (\ref{eq:fact})
to efficiently solve optimization problems that are defined on the
elliptope (\ref{eq:elliptope}). Whereas the algorithms in
\cite{Grubisic07} evolve on the \emph{Cholesky manifold}---a
submanifold of $\rr^{n\times p}$ whose intersection with almost
all equivalence classes is a singleton---, the methods proposed
here work conceptually on the entire quotient space and
numerically in $\rr^{n\times p}$, using the machinery of
Riemannian submersions.\\

The paper is organized as follows. Section \ref{sec:opt_cond}
derives conditions for an optimizer of (\ref{eq:2}) to represent a
solution of the original problem (\ref{eq:1}). A meta-algorithm
for solving (\ref{eq:1}) based on the factorization
(\ref{eq:fact}) is built upon these theoretical results. Section
\ref{sec:manifold_optim} describes the geometry of the underlying
quotient manifold and proposes an algorithm for solving
(\ref{eq:2}) based on second order derivative information.
Sections \ref{sec:max_cut} and \ref{sec:SPCA} illustrate the new
approach on two applications: the maximal cut of a graph and the
sparse principal component analysis problem.

\section{Notations}
Given a function $f: {\mathbb{S}}^{n}\rightarrow \mathbb{R} : X
\mapsto f(X),$ we define the function
\[\tilde{f}:
\mathbb{R}^{n \times p}\rightarrow \mathbb{R} : Y \mapsto
\tilde{f}(Y)= f(Y Y^T).\] The operator $\nabla \cdot$ stands for
the first order derivative, i.e., the matrix $B=\nabla_X f(X_0)$
represents the gradient of $f$ with respect to the variable $X$
evaluated at the point $X_0$. $f$ is assumed to be differentiable
and $B$ is defined element wise by
\[B_{i,j}=\frac{\partial f}{\partial X_{i,j}}(X_0).\]Finally,
\[ D_X f(X_0)[Z] = \underset{t \rightarrow 0}{\lim}\frac{f(X_0+t
Z)-f(X_0)}{t},\] denotes the derivative with respect to $X$ of the
function $f$ at the point $X_0$ in the direction $Z$. It holds
that \[D_X f(X_0)[Z] = \langle\nabla_X f(X_0),Z \rangle, \] where
$\langle \cdot , \cdot \rangle $ denotes the Frobenius inner
product $\langle Z_1 , Z_2 \rangle=\mathrm{Tr}(Z_1^T Z_2) $.

\section{Optimality conditions}
\label{sec:opt_cond} This section derives and analyzes the
optimality conditions of both programs (\ref{eq:1}) and
(\ref{eq:2}). These provide theoretical insight about the rank $p$
at which (\ref{eq:2}) should be solved as well as conditions for
an optimizer of (\ref{eq:2}) to represent a solution of the
original problem (\ref{eq:1}). A meta-algorithm for solving
(\ref{eq:1}) is then derived from these results.

\subsection{First-order optimality conditions}
\begin{lmm}
\label{lmm:1} A symmetric matrix $X \in {\mathbb{S}}^{n}$ solves
(\ref{eq:1}) if and only if there exist a vector $\sigma\in
\mathbb{R}^m$ and a symmetric matrix $S \in {\mathbb{S}}^{n}$ such
that the following holds,
\begin{equation}
\label{eq:KKT1}
\begin{array}{l}
\mathrm{Tr}(A_i X)=b_i, \\
X \succeq 0,\\
S \succeq 0,\\
S X=0,\\
S=\nabla_X f(X) - \sum_{i=1}^m  \sigma_i A_i.
\end{array}
\end{equation}
\end{lmm}
\begin{proof}
These are the first order KKT-conditions, which are necessary and
sufficient in case of convex programs \cite{boyd04}.
\end{proof}\\

\begin{lmm}
\label{lmm:2} If $Y$ is a local optimum of (\ref{eq:2}), then
there exists a vector $\lambda \in \mathbb{R}^m$ such that
\begin{equation}
\label{eq:KKT2}
\begin{array}{l}
\mathrm{Tr}(Y^T A_i Y)=b_i, \\
(\nabla_{X} f(Y Y^T) - \sum_{i=1}^m \lambda_i A_i) Y = 0.
\end{array}
\end{equation}
If the $\{A_i Y\}_{i=1,\ldots,m}$ are linearly independent, the
vector $\lambda$ is unique.
\end{lmm}
\begin{proof}
These are the first order KKT-conditions for the program
(\ref{eq:2}).
\end{proof}\\

Given a local minimizer $Y$ of (\ref{eq:2}), one readily notices
that all but one condition of Lemma \ref{lmm:1} hold for the
symmetric positive semidefinite matrix $Y Y^T$. Comparison of
Lemma 1 and Lemma 2 therefore provides the following relationship
between the nonconvex program (\ref{eq:2}) and the convex program
(\ref{eq:1}).
\begin{thrm}
\label{thm:1} A local minimizer $Y$ of the nonconvex program
(\ref{eq:2}) provides the solution $Y Y^T$ of the convex program
(\ref{eq:1}) if and only if the matrix
\begin{equation} \label{eq:dual_S} S_Y=\nabla_{X} f(Y Y^T) -
\sum_{i=1}^m \lambda_i A_i\end{equation} is positive semidefinite
for the Lagrangian multipliers $\lambda_i$ that satisfy
(\ref{eq:KKT2}).
\end{thrm}
\begin{proof}
Check the conditions of Lemma \ref{lmm:1} for the tuple
$\{X,S,\sigma\}=\{YY^T,S_Y,\lambda\}$.
\end{proof}\\

It is important to note that, under Assumption 2, the Lagrangian
multipliers in (\ref{eq:KKT2}) have the closed-form expression,
\begin{equation}
\label{eq:lag_mult}\lambda_i=\frac{\mathrm{Tr}(Y^T A_i \nabla_{X}
f(Y Y^T) Y)}{\mathrm{Tr}(Y^T A_i^2 Y)}.
\end{equation}
Hence, the dual matrix $S_Y$ in (\ref{eq:dual_S}) can be
explicitly evaluated at an optimizer $Y$ of (\ref{eq:2}).

\subsection{Second-order optimality conditions}
Let $\mathcal{L}(Y,\lambda)$ denote the Lagrangian of the
nonconvex program (\ref{eq:2}), i.e.,
\[\mathcal{L}(Y,\lambda)=f(Y Y^T)- \sum_{i=1}^m \lambda_i (\mathrm{Tr}(Y^T A_i
Y)-b_i).
\]
In the following, the Lagrangian multipliers $\lambda$ are assumed
to satisfy (\ref{eq:KKT2}). A necessary condition for $Y \in
\mathbb{R}^{n \times p}$ to be optimal is that it is a critical
point, i.e., $\nabla_Y \mathcal{L}(Y,\lambda)~=~0$.
\begin{lmm}
\label{lmm:3}For a minimizer $Y \in \mathbb{R}^{n \times p}$ of
(\ref{eq:2}), one has
\begin{equation}
\label{eq:KKT2_2nd} \mathrm{Tr}(Z^T D_Y
\nabla_Y\mathcal{L}(Y,\lambda)[Z]) \geq 0
\end{equation}
for any matrix $Z \in \mathbb{R}^{n \times p}$ that satisfies,
\begin{equation}
\label{eq:KKT2_2nd2} \mathrm{Tr}(Z^T A_i Y) = 0, \; i=1,\ldots, m.
\end{equation}
\end{lmm}
\begin{proof}
These are the second order KKT-conditions of the program
(\ref{eq:2}).
\end{proof}\\

\begin{lmm}
\label{lmm:4} Because of the convexity of $f(X)$, one always has
\begin{equation}\label{eq:alpha}\frac{1}{2}\mathrm{Tr}(Z^T D_Y \nabla_Y\mathcal{L}(Y,\lambda)[Z])=\mathrm{Tr}(Z^T S_Y
Z) + \alpha\end{equation} with $\alpha \geq 0$ and for any matrix
$Z$ that satisfies (\ref{eq:KKT2_2nd2}).
 The term
$\alpha$ cancels out once $Y Z^T=0$.
\end{lmm}
\begin{proof}
By noting that $\nabla_Y \mathcal{L}(Y,\lambda) = 2 S_Y Y$, one
has
\begin{multline}
\label{eq:lemma4} \frac{1}{2}\mathrm{Tr}(Z^T D_Y
\nabla_Y\mathcal{L}(Y,\lambda)[Z])=\\ \mathrm{Tr}(Z^T S_Y Z) +
\mathrm{Tr}(Z^T D_Y (\nabla_{X} f(Y Y^T)) [Z] Y) - \sum_{i=1}^m
D_Y \lambda_i [Z] \mathrm{Tr} (Z^T A_i Y).
\end{multline}
The last term of (\ref{eq:lemma4}) cancels out by virtue of
(\ref{eq:KKT2_2nd2}) and the convexity of the function $f(X)$
ensures the second term of (\ref{eq:lemma4}) to be nonnegative,
i.e.,
\begin{align*}
\mathrm{Tr}(Z^T D_Y(\nabla_{X} f(Y Y^T)) [Z] Y) & = \frac{1}{2}
\mathrm{Tr}((Y Z^T + Z Y^T) D_Y(\nabla_{X} f(Y Y^T)) [Z])\\
 & = \frac{1}{2} \mathrm{Tr}(W^T D_X(\nabla_{X} f) [W])\\
 & \geq 0,
\end{align*}
where $X=Y Y^T$ and $W= Y Z^T + Z Y^T \in S_n$.
\end{proof}\\

\begin{thrm}
\label{thm:2} A local minimizer $Y$ of the program (\ref{eq:2})
provides the solution $X=Y Y^T$ of the program (\ref{eq:1}) if it
is rank deficient.
\end{thrm}
\begin{proof}
For the matrix $Y \in \mathbb{R}^{n \times p}$ to span a
$r$-dimensional subspace, the following factorization has to hold,
\begin{equation}
\label{eq:proof4} Y=\tilde{Y} M^T,
\end{equation}
with $\tilde{Y} \in \mathbb{R}^{n \times r}$ and $M$ a full rank
matrix of $\mathbb{R}^{p \times r}$. Let $M_{\perp} \in
\mathbb{R}^{p \times (p-r)}$ be an orthogonal basis for the
orthogonal complement of the column space of M, i.e., $M^T
M_{\perp} =0$ and $M_{\perp}^T M_{\perp}=I$. For any matrix
$\tilde{Z} \in \mathbb{R}^{n \times (p-r)}$, the matrix $Z=
\tilde{Z} M_{\perp}^T$ satisfies \[ Y Z^T = 0\] such that the
conditions (\ref{eq:KKT2_2nd2}) hold and $\alpha$ cancels out in
(\ref{eq:alpha}). Thus, by virtue of Lemmas \ref{lmm:3} and
\ref{lmm:4},  \[\mathrm{Tr}(Z^T S_Y Z) \geq 0,\] for matrices
$Z=\tilde{Z} M_{\perp}^T$, i.e., the matrix $S_Y$ is positive
semidefinite and $X=Y Y^T$ is a solution of the problem
(\ref{eq:1}).
\end{proof}\\

\begin{crllr}
\label{crllr:1} In the case $p=n$, any local minimizer $Y \in
\mathbb{R}^{n \times n}$ of the program (\ref{eq:2}) provides the
solution $X=Y Y^T$ of the program (\ref{eq:1}).
\end{crllr}
\begin{proof}
If $Y$ is rank deficient, the matrix $X=Y Y^T$ is optimal for
(\ref{eq:1}) by virtue of Theorem \ref{thm:2}. Otherwise, the
matrix $S_Y$ is zero because of the second condition in
(\ref{eq:KKT2}) and $X$ is optimal for (\ref{eq:1}).
\end{proof}\\

\subsection{An algorithm to solve the convex problem}
The proposed algorithm consists in solving a sequence of nonconvex
problems (\ref{eq:2}) of increasing dimension until the resulting
local minimizer $Y$ represents a solution of the convex program
(\ref{eq:1}). Both Theorems \ref{thm:1} and \ref{thm:2} provide
conditions to check this fact. When the program (\ref{eq:2}) is
solved in a dimension $p$ smaller than the unknown rank $r$, none
of these conditions can be fulfilled. The dimension $p$ is thus
incremented after each resolution of (\ref{eq:2}). In order to
ensure a monotone decrease of the cost function through the
iterations, the optimization algorithm that solves (\ref{eq:2}) is
initialized with a matrix corresponding to $Y$ with an additional
zero column appended, i.e., $Y_0=[Y | 0]$. Since this
initialization occurs when the local minimizer $Y \in
\mathbb{R}^{n \times p}$ of (\ref{eq:2}) does not represent the
solution of (\ref{eq:1}), $Y_0$ is a saddle point of the nonconvex
problem for the dimension $p+1$. This can be a critical issue for
many optimization algorithms. Fortunately, in the present case, a
descent direction from $Y_0$ can be explicitly evaluated. For
Lemma \ref{lmm:4}, the matrix $Z= [0 |v]$, for instance, where $0$
is a zero matrix of the size of $Y$ and $v$ is the eigenvector of
$S_Y$ related to the smallest algebraic eigenvalue verifies,
\[\frac{1}{2} \mathrm{Tr}(Z^T D_Y\nabla_{Y}\mathcal{L}(Y_0,\lambda)[Z])= v^T S_Y v \leq 0,\]
since $Y_0 Z^T=0$ for the Lagrangian multipliers $\lambda$ given
in (\ref{eq:lag_mult}). All these elements lead to the
meta-algorithm displayed in Algorithm \ref{algo1}. The parameter
$\varepsilon$ fixes a threshold on the eigenvalues of $S_Y$ to
decide about the
nonnegativity of this matrix. $\varepsilon$ is chosen to $10^{-12}$ in our implementation.\\

\begin{algorithm}[h]
\dontprintsemicolon \SetKwInOut{Input}{input}
\SetKwInOut{Output}{output}
 \Input{Initial rank $p_0$, initial iterate $Y^{(0)} \in \mathbb{R}^{n \times p_0}$ and parameter $\varepsilon$.}
 \Output{The solution $X$ of the convex program (\ref{eq:1}).}
 \Begin{$p \longleftarrow p_0$\\
 $Y_p \longleftarrow Y^{(0)}$\\
 $\mathrm{stop} \longleftarrow 0$\\
 \While{$\mathrm{stop} \neq 1$}{Initialize an optimization scheme with $Y_p$ to find a local minimum $Y_p^*$ of (\ref{eq:2}) by exploiting
 a descent direction $Z_p$ if available.\\
 \eIf{$p = p_0 \; \mathbf{and} \; \mathrm{rank}(Y_p^*)<p$}{$\mathrm{stop}=1$}{Find the smallest eigenvalue $\lambda_{\min}$ and the related eigenvector $V_{\min}$ of the matrix $S_Y$ (\ref{eq:dual_S}).\\
 \eIf{$\lambda_{\min}\geq - \varepsilon$}{$\mathrm{stop}=1$}{$p \longleftarrow p+1$ \\ $Y_p \longleftarrow [Y_p^* | 0]$\\
A descent direction from the saddle point $Y_p$ is given by $Z_p=
[0| V_{\min}].$}} } $X \longleftarrow Y_p^* Y_p^{*T}$}
 \label{algo1} \caption{Meta-algorithm for solving the convex program (\ref{eq:1})  \protect\refstepcounter{footnote}\protect\footnotemark[\thefootnote] }
\end{algorithm}
\footnotetext[\thefootnote]{A Matlab implementation of Algorithm
\ref{algo1} with the manifold-based optimization method of Section
\ref{sec:manifold_optim} can be downloaded from
http://www.montefiore.ulg.ac.be/$\sim$journee.}

It should be mentioned that, to check the optimality for the
convex program (\ref{eq:1}) of a local minimizer $Y_{p}^*$, the
rank condition of Theorem \ref{thm:2} is computationally cheaper
to evaluate than the nonnegativity condition of Theorem
\ref{thm:1}. Nevertheless, the rank condition does not provide a
descent direction to escape saddle points. It furthermore requires
to solve the program (\ref{eq:2}) at a dimension that is strictly
greater than $r$, the rank of the solution of (\ref{eq:1}). Hence,
this condition is only used at the initial rank $p_0$ and holds if
$p_0$ is chosen larger than the unknown $r$. Numerically, the rank
of $Y_{p_0}^*$ is computed as the number of singular values that
are greater than a threshold fixed at $10^{-6}$. The algorithm
proposed in \cite{Burer03} exploits exclusively the rank condition
of Theorem \ref{thm:2}. For this reason, each optimization of
(\ref{eq:2}) has to be randomly initialized and the algorithm in
\cite{Burer03} is not a descent algorithm.\\

By virtue of Corollary \ref{crllr:1}, Algorithm \ref{algo1} stops
at the latest once $p=n$. The applications proposed in Sections
\ref{sec:max_cut} and \ref{sec:SPCA} indicate that in practice,
however, the algorithm stops at a rank $p$ that is much lower than
the dimension $n$. If $p_0 < r$, then the algorithm stops once $p$
equals the rank $r$ of the solution of (\ref{eq:1}). These
applications also illustrate that the magnitude of smallest
eigenvalue $\lambda_{\min}$ of the matrix $S_Y$ can be used to
monitor the convergence. The value $|\lambda_{\min}|$ indicates
whether the current iterate is close to satisfy the KKT conditions
(\ref{eq:KKT1}). This feature is of great interest once an
approximate solution to (\ref{eq:1}) is sufficient. The threshold
$\varepsilon$ set on $\lambda_{\min}$ controls then the accuracy
of the result.\\

A trust-region scheme based on second-order derivative information
is proposed in the next section for computing a local minimum of
(\ref{eq:2}). This method is provided with a convergence theory
that ensures the iterates to converge towards a local
minimizer.\\

Hence, the proposed algorithm presents the following notable
features. First, it converges toward the solution of the convex
program (\ref{eq:1}) by ensuring a monotone decrease of the cost
function. Then, the magnitude of the smallest eigenvalue of $S_Y$
provides a mean to monitor the convergence. Finally, the inner
problem (\ref{eq:2}) is solved by second-order methods featuring
superlinear local convergence.

\section{Manifold-based optimization}
\label{sec:manifold_optim} We now derive an optimization scheme
that solves the nonconvex and nonlinear program,
\begin{equation}
\label{eq:22}
\begin{array}{ll}
\underset{Y \in \mathbb{R}^{n \times p}}{\min} & \tilde{f}(Y)\\
\text{s.t.} & \mathrm{Tr}(Y^T A_i Y)=b_i, \; A_i \in {\mathbb{S}}^{n}, b_i \in \mathbb{R},\; i=1,\ldots,m,\\
\end{array}
\end{equation}
where $\tilde{f}(Y)=f(Y Y^T)$ for some $f: \mathbb{S}^n \rightarrow \mathbb{R}$.\\

As previously mentioned, Program (\ref{eq:22}) is invariant by
right-multiplication of the variable $Y$ by orthogonal matrices.
The critical points of (\ref{eq:22}) are thus non isolated. The
proposed algorithm exploits this symmetry by optimizing the cost
$\tilde{f}(\cdot)$ on the quotient
\[\mathcal{M} = \bar{\mathcal{M}} / \mathcal{O}_p,\] where
$\mathcal{O}_p=\{Q\in \mathbb{R}^{p \times p} | Q^T Q = I \}$ is
the orthogonal group and $\bar{\mathcal{M}}=\{Y \in
\mathbb{R}_{*}^{n \times p} : \mathrm{Tr}(Y^T A_i Y)=b_i, \;
i=1,\ldots,m\}$ is the feasible set.\footnote{$\mathbb{R}_*^{n
\times p}$ is the noncompact Stiefel manifold of full-rank $n
\times p$ matrices. The nondegeneracy condition is required to
deal with differentiable manifolds.} Each point of the quotient
$\mathcal{M}$ is an equivalence class
\begin{equation} \label{eq:eq_class}[Y]=\{ Y Q |
Q \in \mathcal{O}_p\}.\end{equation} It can be proven that the
quotient $\mathcal{M}$ presents a manifold structure
\cite{AbsMahSep2008}. Program (\ref{eq:22}) is thus strictly
equivalent to the optimization problem,
\[ \underset{[Y]\in \mathcal{M}}{\min} \bar f([Y]),\]
for the function $\bar{f}: \mathcal{M} \rightarrow \mathbb{R} :
[Y] \mapsto \bar{f}([Y])= \tilde{f}(Y).$\\

Several unconstrained optimization methods have been generalized
to search spaces that are differentiable manifolds. This is, e.g.,
the case of
 the trust-region approach. Details on this algorithm can be found in
\cite{Absil07,AbsMahSep2008}. It is important to mention that this
algorithm is provided with a convergence theory whose results are
similar to the ones related to classical unconstrained
optimization. In particular, trust-region methods on manifolds
converge globally to stationary points of the cost function if the
inner iteration produces a model decrease that is better than a
fixed fraction of the Cauchy decrease; such a property is
achieved, e.g., by the Steihaug-Toint inner iteration. Since the
iteration is moreover a descent method, convergence to saddle
points or local maximizers is not observed in practice. It is
possible to obtain guaranteed convergence to a point where the
second-order necessary conditions of optimality hold, by using
inner iterations that exploit the model more fully (e.g., the
inner iteration of Mor\'e and Sorensen), but these inner
iterations tend to be prohibitively expensive for large-scale
problems. For appropriate choices of the inner iteration stopping
criterion, trust-region methods converge locally superlinearly
towards the nondegenerate local minimizers of the cost function.
The parameter $\theta$ in Equation (10) of \cite{Absil07} has been
set to one, which guarantees a quadratic
convergence.\\

A few important objects have to be specified to exploit the
trust-region algorithm of \cite{Absil07} in the present context.
First, the tangent space at a point $Y$ of the manifold
$\bar{\mathcal{M}}$,
\[ T_Y \bar{\mathcal{M}}=\{Z \in \mathbb{R}^{n \times p} : \mathrm{Tr}(Y^T A_i
Z)=0, \; i=1, \ldots, m \},\] has to be decomposed in two
orthogonal subspaces, the \emph{vertical space}
$\mathcal{V}_Y\mathcal{M}$ and the \emph{horizontal space}
$\mathcal{H}_Y\mathcal{M}$. The vertical space
$\mathcal{V}_Y\mathcal{M}$ corresponds to the tangent space to the
equivalence classes,
\[
\mathcal{V}_Y\mathcal{M}= \{Y \Omega : \Omega \in R^{p \times p},
\; \Omega^T=-\Omega\}.
\]
The horizontal space $\mathcal{H}_Y\mathcal{M}$ is the orthogonal
complement of $\mathcal{V}_Y\mathcal{M}$ in $T_Y
\bar{\mathcal{M}}$, i.e., \begin{equation} \label{eq:h_space}
\mathcal{H}_Y\mathcal{M}=\{Z \in T_Y \bar{\mathcal{M}}: Z^T Y= Y^T
Z\},
\end{equation}
for the Euclidean metric $ \langle Z_1,Z_2
\rangle=\mathrm{Tr}(Z_1^T Z_2)$ for all $Z_1, Z_2 \in T_Y
\bar{\mathcal{M}}$. Expression (\ref{eq:h_space}) results from the
equality $\mathrm{Tr}(S \Omega) =0$ that holds for any symmetric
matrix $S$ and skew-symmetric matrix $\Omega$ of
compatible dimension.\\

Let $N_Y\bar{\mathcal{M}}$, the \emph{normal space} to
$\bar{\mathcal{M}}$ at $Y$, denote the orthogonal complement of
$T_Y\bar{\mathcal{M}}$ in $\mathbb{R}^{n\times p}$, i.e.,
$N_Y\bar{\mathcal{M}} = \{\sum_{i=1}^m \alpha_i A_i Y, \; \alpha
\in \mathbb{R}^m\}.$ Hence, the Euclidean space $\mathbb{R}^{n
\times p}$ can be divided into three mutually orthogonal
subspaces,
\[\mathbb{R}^{n \times p} = \mathcal{H}_Y\mathcal{M} \oplus \mathcal{V}_Y\mathcal{M} \oplus N_Y\bar{\mathcal{M}}.\]
The trust-region algorithm proposed in \cite{Absil07} requires a
projection $P_Y(\cdot)$ from $\mathbb{R}^{n \times p}$ to
$\mathcal{H}_Y \mathcal{M}$ along $\mathcal{V}_Y\mathcal{M} \oplus
N_Y\bar{\mathcal{M}}$. The following theorem provides a
closed-form expression.
\begin{thrm}
Let $Y$ be a point on $\bar{\mathcal{M}}$. For a matrix $Z \in
\mathbb{R}^{n \times p}$, the projection $P_Y(\cdot):\mathbb{R}^{n
\times p} \rightarrow \mathcal{H}_Y\mathcal{M}$ is given by
\[P_Y(Z)=Z-Y \Omega-\sum_{i=1}^m \alpha_i A_i Y,\] where $\Omega$ is the skew symmetric matrix that solves
the Sylvester equation
\begin{equation*}
\Omega Y^T Y+Y^T Y \Omega = Y^T Z-Z^T Y,
\end{equation*}
and with the coefficients
\begin{equation*} \alpha_i=\frac{\mathrm{Tr}(Z^T A_i Y)}{\mathrm{Tr}(Y^T
A_i^2 Y)}.\end{equation*}
\end{thrm}
\begin{proof}
Any vector $Z \in \mathbb{R}^{n \times p}$ presents a unique
decomposition
\[Z = Z_{\mathcal{V}_Y\mathcal{M}} + Z_{\mathcal{H}_Y\mathcal{M}} +
Z_{N_Y\bar{\mathcal{M}}},
\]
where each element $Z_\mathcal{X}$ belongs to the Euclidean space
$\mathcal{X}$. The orthogonal projection $\mathcal{P}_Y(\cdot)$
extracts the component that lies in the horizontal space, i.e.,
\begin{equation*} P_Y(Z)=Z-Y \Omega -
\sum_{i=1}^m \alpha_i A_i Y,\end{equation*} with $\Omega$ a skew
symmetric matrix. The parameters $\Omega$ and $\alpha$ are
determined from the linear equations
\begin{align*}
& Y^T P_Y(Z) = P_Y(Z)^T Y,\\
& \mathrm{Tr}(Y^T A_i P_Y(Z)) = 0,  \quad i=1 \ldots m,
\end{align*}
which are satisfied by any element of the horizontal space.
\end{proof}\\

The projection $P_Y(\cdot)$ provides simple formulas to compute
derivatives of the function $\bar{f}$ (defined on the quotient
manifold) from derivatives of the function $\tilde{f}$ (defined in
the Euclidean space). The gradient corresponds to the projection
on the horizontal space of the gradient of the function
$\tilde{f}(Y)$, i.e.,
\[\mathrm{grad}\bar{f}(Y)=P_Y ( \nabla_Y \tilde{f}(Y)).\] The
Hessian applied on a direction $Z \in \mathcal{H}_Y\mathcal{M}$ is
given by \[\mathrm{Hess}\bar{f}(Y)[Z]=P_Y(D_Y(P_Y ( \nabla_Y
\tilde{f}(Y)))[Z]),\] where the directional derivative
$D_Y(\cdot)[\cdot]$ is performed in the Euclidean space $R^{n
\times
p}$.\\

Finally, a last ingredient needed by the trust-regions algorithm
in \cite{Absil07} is a \emph{retraction} $\mathcal{R}_Y(\cdot)$
that maps a search direction $Z$ (an element of the horizontal
space at $Y$) to a matrix representing a new point on the manifold
$\mathcal{M}$. Such a mapping is for example given by the
projection of the matrix $\tilde{Y}=Y+Z$ along the Euclidean space
$N_Y\bar{\mathcal{M}}$, i.e.,
 \begin{equation} \label{eq:retraction}
\mathcal{R}_Y(Z)= [\tilde{Y} + \sum_{i=1}^m \alpha_i A_i
\tilde{Y}],\end{equation} where $[\cdot]$ denotes the equivalence
class (\ref{eq:eq_class}) and the coefficients $\alpha_i$ are
chosen such that
\[\mathrm{Tr}(\bar{Y}^T A_i \bar{Y})=b_i,\]
with $\bar{Y}=\tilde{Y} + \sum_{i=1}^m \alpha_i A_i \tilde{Y}$.
Under Assumption 2, the coefficients $\alpha_i$ are easily
computed as the solution of the quadratic polynomial,
\[ \alpha_i^2 \mathrm{Tr}(\tilde{Y}^T A_i^3 \tilde{Y}) + 2 \alpha_i
\mathrm{Tr}(\tilde{Y}^T A_i^2 \tilde{Y}) + \mathrm{Tr}(\tilde{Y}^T
A_i \tilde{Y})=b_i.
\]
In case of the elliptope $\mathcal{E}$, Equation
(\ref{eq:retraction}) becomes,
\[\mathcal{R}_Y(Z)= [\mathrm{Diag}((Y+Z) (Y+Z)^T)^{-\frac{1}{2}} (Y+Z)],\]
where $\mathrm{Diag}(X)$ denotes the diagonal matrix whose
diagonal elements are those of $X$ and the brackets refer to the
equivalence class (\ref{eq:eq_class}). For the spectahedron
$\mathcal{S}$, the retraction (\ref{eq:retraction}) is given by
\[\mathcal{R}_Y(Z)=\left[\frac{Y+Z}{\sqrt{\mathrm{Tr}((Y+Z)^T (Y+Z))}}\right].
\]

The complexity of the manifold-based trust-region algorithm in the
context of program (\ref{eq:22}) is dominated by the computational
cost required to evaluate the objective $\tilde{f}(Y)$, the
gradient $\nabla_Y \tilde{f}(Y)$ and the directional derivative
$D_Y(\nabla_Y \tilde{f}(Y))[Z]$. Hence, the costly operations are
performed in the Euclidean space $R^{n \times p}$, whereas all
manifold-related operations, such as evaluating a metric, a
projection and a retraction, are of linear complexity with the
dimension $n$.

\section{Optimization on the elliptope: the max-cut SDP relaxation}
\label{sec:max_cut} A first application of the proposed
optimization method concerns the maximal cut of a graph.\\

The maximal cut of an undirected and weighted graph corresponds to
the partition of the vertices in two sets such that the sum of the
weights associated to the edges crossing between these two sets is
the largest. Computing the maximal cut of a graph is a NP-complete
problem. Several relaxations to that problem have been proposed.
The most studied one is the 0.878-approximation algorithm
\cite{Goemans_95} that solves the following semidefinite program
(SDP),
\begin{equation}
\label{eq:max_cut}
\begin{array}{ll}
\underset{X \in {\mathbb{S}}^{n}}{\min} &  \mathrm{Tr}(A X)\\
\text{s.t.} & \mathrm{diag}(X)=\mathbf{1},\\
& X \succeq 0, \\
\end{array}
\end{equation}
where $A=-\frac{1}{4} L$ with $L$ the Laplacian matrix of the
graph and $\mathbf{1}$ is a vector of all ones. This relaxation is
tight in case of a rank-one solution.\\

As previously mentioned, the elliptope,
\[\mathcal{E}=\{ X \in
{\mathbb{S}}^{n}|X\succeq 0,\mathrm{diag}(X)=\mathbf{1}\},\]
satisfies Assumption 2. Hence, Program (\ref{eq:max_cut}) is a
good candidate for the proposed framework. Using the factorization
$X=Y Y^T$, the optimization problem is defined on the quotient
manifold
$\mathcal{M}_{\mathcal{E}}=\bar{\mathcal{M}}_{\mathcal{E}}/\mathcal{O}_p$,
where
\[\bar{\mathcal{M}}_{\mathcal{E}}=\{Y \in \mathbb{R}_{*}^{n
\times p} : \mathrm{diag}(Y Y^T)=\mathbf{1} \}.\] The complexity
of Algorithm \ref{algo1} in the present context is of order $O(n^2
p)$. This complexity is dominated by both the manifold-based
optimization and the eigenvalue decomposition of the dual variable
$S_Y$, that are $O(n^2 p)$. The computational cost related to the
manifold-based optimization is
however reduced in case of matrices $A$ that are sparse.\\

Table \ref{tbl:max_cut} presents computational results obtained
with Algorithm \ref{algo1} for computing the maximal cut of a set
of graphs. The parameter $n$ denotes the number of vertices of
these graphs and corresponds thus to the size of the variable $X$
in (\ref{eq:max_cut}). More details on these graphs can be found
in \cite{Burer03} and references therein. The low-rank method is
compared with the SDPLR algorithm proposed in \cite{Burer03}, that
also exploits the low rank factorization $X=Y Y^T$ in the case of
semidefinite programs (SDP). The rank of the optimizer $Y^*$
indicates that the factorization $X=Y Y^T$ reduces significantly
the size of the search space. Concerning the computational time,
it is important to mention that Algorithm \ref{algo1} has been
implemented in Matlab, whereas a C implementation of the SDPLR
algorithm has been provided by the authors of \cite{Burer03}.
Although this renders a rigorous comparison of the computational
load difficult, Table \ref{tbl:max_cut} suggests
that both methods perform similarly.\\

\begin{table}[h] \centerline{
 \begin{tabular}{l c c c c c c c c}
\hline
 &  & & & \multicolumn{2}{c}{\textbf{Objective values}} & & \multicolumn{2}{c}{\textbf{CPU time} (sec)}\\
\textbf{Graph} & $n$& $\mathrm{Rank}(Y^*)$ && Algo. \ref{algo1}  & SDPLR & &  Algo. \ref{algo1}& SDPLR\\
\hline
   toruspm3-8-50  & 512  &   8  & &  -527.81  &  -527.81  & &    17  &     3\\
   toruspm3-15-50 & 3375 &  15  & &  -3474.79  &  -3474.76  & &  1051  &   181 \\
   torusg3-8      & 3375 &   7  & &  -3187.61  &  -3188.09  & &   375  &   228 \\
    G1            & 800  &  13  & &  -12083.2  &  -12083.1  & &    57  &    35 \\
    G11           & 800  & 5  & &  -629.16  &  -629.15  & &    53  &    15 \\
    G14           & 800  & 13  & &  -3191.57  &  -3191.53  & &    82  &    13  \\
    G22           & 2000 & 18  & &  -14136.0  &  -14135.9  & &   358  &   101  \\
    G32           & 2000 &  5  & &  -1567.58  &  -1567.57  & &   158  &    69  \\
    G35           & 2000 & 14  & &  -8014.57  &  -8014.33  & &   525  &    68 \\
    G36           & 2000 & 13  & &  -8005.60  &  -8005.80  & &   459  &   115  \\
    G58           & 5000 & 8  & &  -20111.3  &  -20135.4  & &  1881  &  1119  \\
\hline
\end{tabular}}
 \caption{Computational results of Algorithm \ref{algo1} (implemented in Matlab) and the SDPLR algorithm (implemented in C) on various graphs.} \label{tbl:max_cut}
\end{table}

Figure \ref{fig:Fig1} depicts the monotone convergence of the
Algorithm \ref{algo1} for the graph toruspm3-15-50. The number of
iterations is displayed on the bottom abscissa, whereas the top
abscissa stands for the rank $p$. Figure \ref{fig:Fig2} indicates
that the smallest eigenvalue $\lambda_{\min}$ of the dual matrix
$S_Y$ monotonically increases to zero. One notices that the
magnitude of $\lambda_{\min}$ gives some insight on the current
accuracy.

\begin{figure}[h]
\centerline{\includegraphics[height=7cm,keepaspectratio]{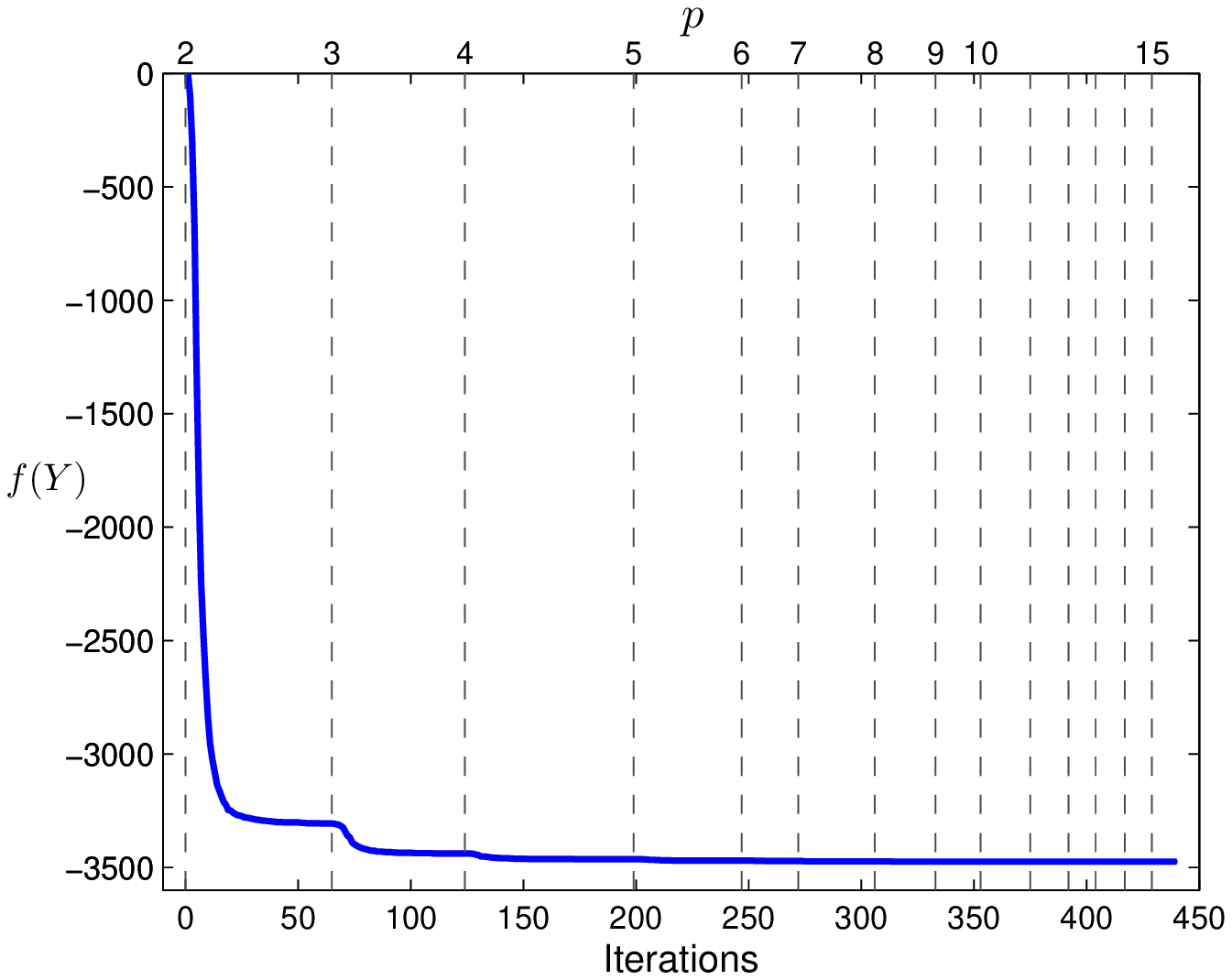}
\includegraphics[height=7cm,keepaspectratio]{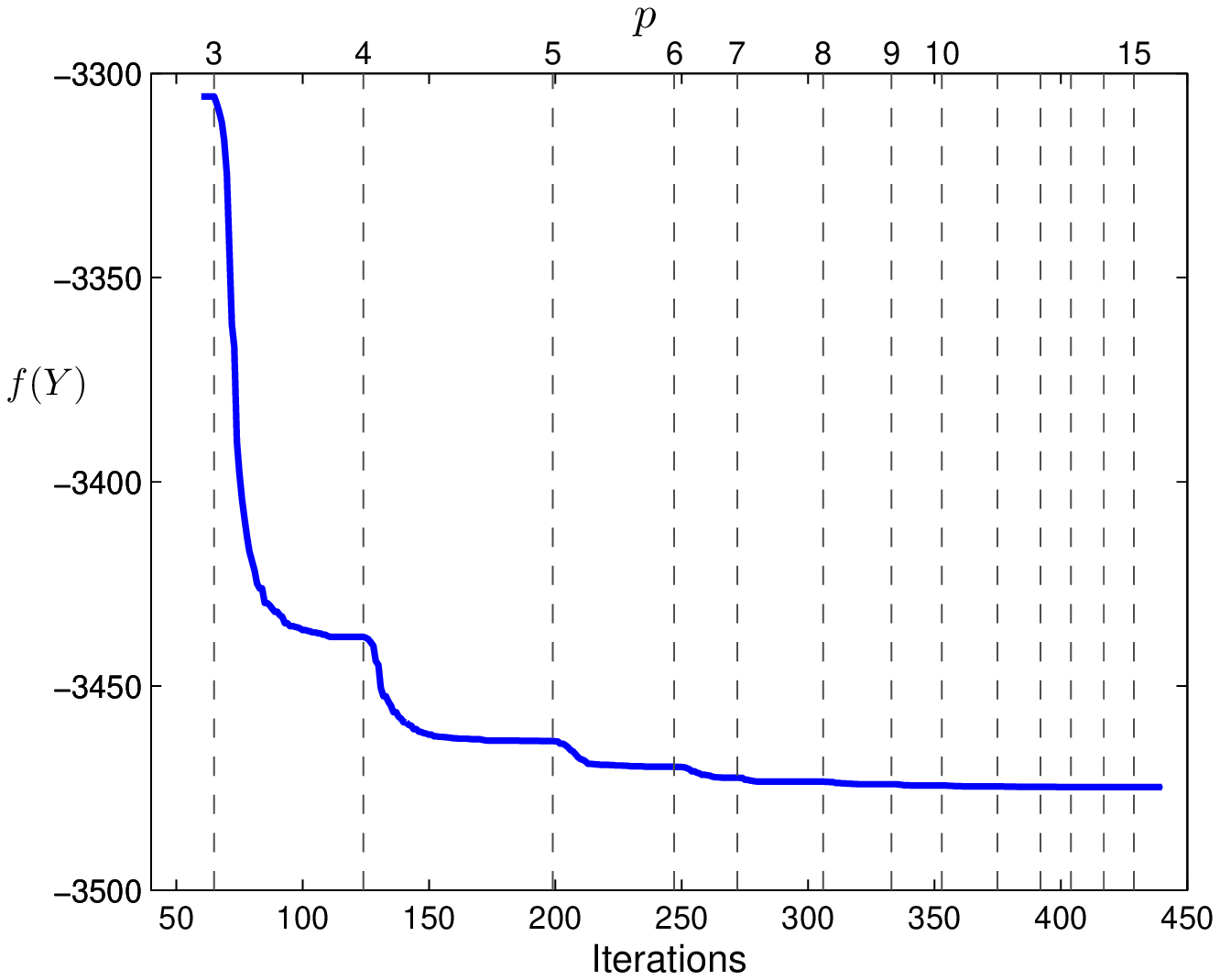}}
\caption{Monotone decrease of the cost function
$f(Y)=\mathrm{Tr}(Y^T A Y)$ through the iterations (bottom
abscissa) and with the rank $p$ (top abscissa) in the case of the
graph toruspm3-15-50.} \label{fig:Fig1}
\end{figure}

\begin{figure}[h]
\centerline{\includegraphics[height=7cm,keepaspectratio]{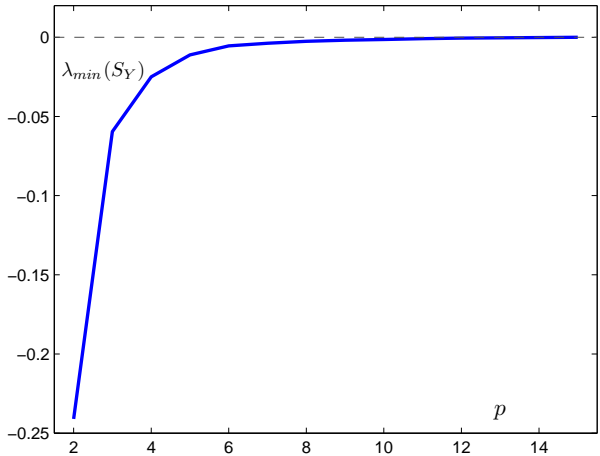}}
\caption{Evolution of the smallest eigenvalue of the matrix $S_Y$
(case of the graph toruspm3-15-50).} \label{fig:Fig2}
\end{figure}

\section{Optimization on the spectahedron: the sparse PCA problem}
\label{sec:SPCA} This section presents three nonlinear programs
that concern the sparse principal component analysis problem and
that can be efficiently solved with the proposed
low-rank optimization approach.\\

Principal component analysis (PCA) is a tool that reduces
multidimensional data to lower dimension. Given a data matrix $A
\in \mathbb{R}^{m \times n}$, the first principal component
consists in the best rank-one approximation of the matrix $A$ in
the least square sense. This decomposition is performed via
estimation of the dominant eigenvector of the empirical covariance
matrix $\Sigma=A^T A$. In many applications, it is of great
interest to get sparse principal components, i.e., components that
yield a good low-rank approximation of $A$ while involving a
limited number of nonzero elements. In case of gene expression
data where the matrix $A$ represents the expression of $n$ genes
through $m$ experiments, getting factors that involve just a few
genes, but still explain a great part of the variability in the
data, appears to be a modelling assumption closer to the biology
than the regular PCA \cite{Teschendorff06}. This tradeoff between
variance and sparsity is the central motivation of sparse PCA
methods. More details on the sparse PCA approach can be found in
\cite{Zou04,Aspremont04} and references therein.\\

Sparse PCA is the problem of finding the unit-norm vector $x \in
\mathbb{R}^n$ that maximizes the Rayleigh quotient of the matrix
$\Sigma=A^T A$ but contains a fixed number of zeros, i.e.,
\begin{equation}
\label{eq:spca_formulation0}
\begin{array}{ll}
\underset{x \in {\mathbb{R}}^{n}}{\max} & x^T \Sigma x\\
\text{s.t.} & x^T x=1, \\ & \mathrm{Card}(x)\leq k,\\
\end{array}
\end{equation}
where $k$ is an integer with $1 \leq k \leq n$ and
$\mathrm{Card}(x)$ is the cardinality of $x$, i.e., the number of
non zero components. Finding the optimal sparsity pattern of the
vector $x$ is of combinatorial complexity. Several algorithms have
been proposed in the literature that find an approximate solution
to (\ref{eq:spca_formulation0}). We refer to \cite{Aspremont04}
for references on these methods. Let us finally mention that the
data matrix $A$ does not necessarily have to present a sparse
pattern. In the context of compressed sensing, for example, one
needs to compute the sparse principal component of a matrix $A$
that is
full and sampled from a gaussian distribution \cite{Aspremont07}.\\

Recently, two convex relaxations have been derived that require to
minimize some nonlinear convex functions on the spectahedron $
\mathcal{S}= \{X \in \mathbb{S}^{m} | X \succeq 0,
\mathrm{Tr}(X)=1\}.$ Both of these relaxations consider a
variation of (\ref{eq:spca_formulation0}), in which the
cardinality appears as a penalty instead of a constraint, i.e.,
\begin{equation}
\label{eq:spca_formulation}
\begin{array}{ll}
\underset{x \in {\mathbb{R}}^{n}}{\max} & x^T \Sigma x- \rho \mathrm{Card}(x)\\
\text{s.t.} & x^T x=1, \\
\end{array}
\end{equation}
with the parameter $\rho \geq 0$.\\

\subsection{A first convex relaxation to the sparse PCA problem}
\label{sec:DSPCA} In \cite{Aspremont04}, Problem
(\ref{eq:spca_formulation}) is relaxed to a convex program in two
steps. First, a convex feasible set is obtained by lifting the
unit norm vector variable $x$ into a matrix variable $X$ that
belongs to the spectahedron, i.e.,
\begin{equation}
\label{eq:dspca1}
\begin{array}{ll}
\underset{X \in {\mathbb{S}}^{n}}{\max} & \mathrm{Tr}( \Sigma X) - \rho \mathrm{Card}(X)\\
\text{s.t.} & \mathrm{Tr}(X)=1,\\ & X \succeq 0.
\end{array}
\end{equation}
The relaxation (\ref{eq:dspca1}) is tight for rank-one matrices.
In such cases, the vector variable $x$ in
(\ref{eq:spca_formulation}) is related to the matrix variable $X$
according to $X=x x^T$. Then, for (\ref{eq:dspca1}) to be convex,
the cardinality penalty is replaced by a convex $l_1$ penalty,
i.e.,
\begin{equation}
\label{eq:dspca2}
\begin{array}{ll}
\underset{X \in {\mathbb{S}}^{n}}{\max} & \mathrm{Tr}( \Sigma X) - \rho \sum_{i,j} |X_{ij}| \\
\text{s.t.} & \mathrm{Tr}(X)=1,\\
& X \succeq 0.
\end{array}
\end{equation}
Finally, a smooth approximation to (\ref{eq:dspca2}) is obtained
by replacing the absolute value by the differentiable function
$h_{\kappa}(x)= \sqrt{x^2+\kappa^2}$ with the parameter $\kappa$
that is very small. A too small $\kappa$ might, however, lead to
ill-conditioned Hessians and thus to numerical problems.\\

The convex program,
\begin{equation}
\label{eq:dspca3}
\begin{array}{ll}
\underset{X \in {\mathbb{S}}^{n}}{\max} & \mathrm{Tr}( \Sigma X) - \rho \sum_{i,j} h_{\kappa}(X_{ij}) \\
\text{s.t.} & \mathrm{Tr}(X)=1,\\
& X \succeq 0,
\end{array}
\end{equation}
fits within the framework (\ref{eq:1}). The variable $X$ is thus
factorized in the product $Y Y^T$ and the optimization is
performed on the quotient manifold
$\mathcal{M}_{\mathcal{S}}=\bar{\mathcal{M}}_{\mathcal{S}}/
\mathcal{O}_p$ where
\[ \bar{\mathcal{M}}_{\mathcal{S}}=\{Y \in \mathbb{R}_{*}^{n \times p} :
\mathrm{Tr}(Y^T Y)=1 \}.\] The computational complexity of
Algorithm \ref{algo1} in the context of program (\ref{eq:dspca3})
is of order $O(n^2 p)$. It should be mentioned that the DSPCA
algorithm derived in \cite{Aspremont04} and that has been tuned to
solve program
(\ref{eq:dspca2}) features a complexity of order $O(n^3)$.\\

Figure \ref{fig:DSPCA1} illustrates the monotone convergence of
Algorithm \ref{algo1} on a random gaussian matrix $A$ of size $50
\times 50$. The sparsity weight factor $\rho$ has been chosen to 5
and the smoothing parameter $\kappa$ equals $10^{-4}$. The maximum
of the nonsmooth cost function in (\ref{eq:dspca2}) has been
computed with the DSPCA algorithm \cite{Aspremont04}. One first
notices that the smooth approximation in (\ref{eq:dspca3})
slightly underestimates the nonsmooth cost function
(\ref{eq:dspca2}). The maximizers of both (\ref{eq:dspca2}) and
(\ref{eq:dspca3}) are, however, almost identical. Then, we should
mention that all numerical experiments performed with the DSPCA
algorithm for solving (\ref{eq:dspca2}) resulted in a rank one
matrix. So, the solution of (\ref{eq:dspca3}) is expected to be
close to rank one. This explains why the improvement in terms of
objective value is very small for ranks larger than one. A
heuristic to speed up the computations would thus consist in
computing an approximate rank one solution of (\ref{eq:dspca3}),
i.e., Algorithm \ref{algo1} is stopped after the iteration $p=1$.
Finally, on the right hand plot, Figure \ref{fig:DSPCA1}
highlights the smallest eigenvalue $\lambda_{\min}$ of the matrix
$S_Y$ as a way to monitor the
convergence.\\
\begin{figure}[h]
\centerline{\includegraphics[height=7cm,keepaspectratio]{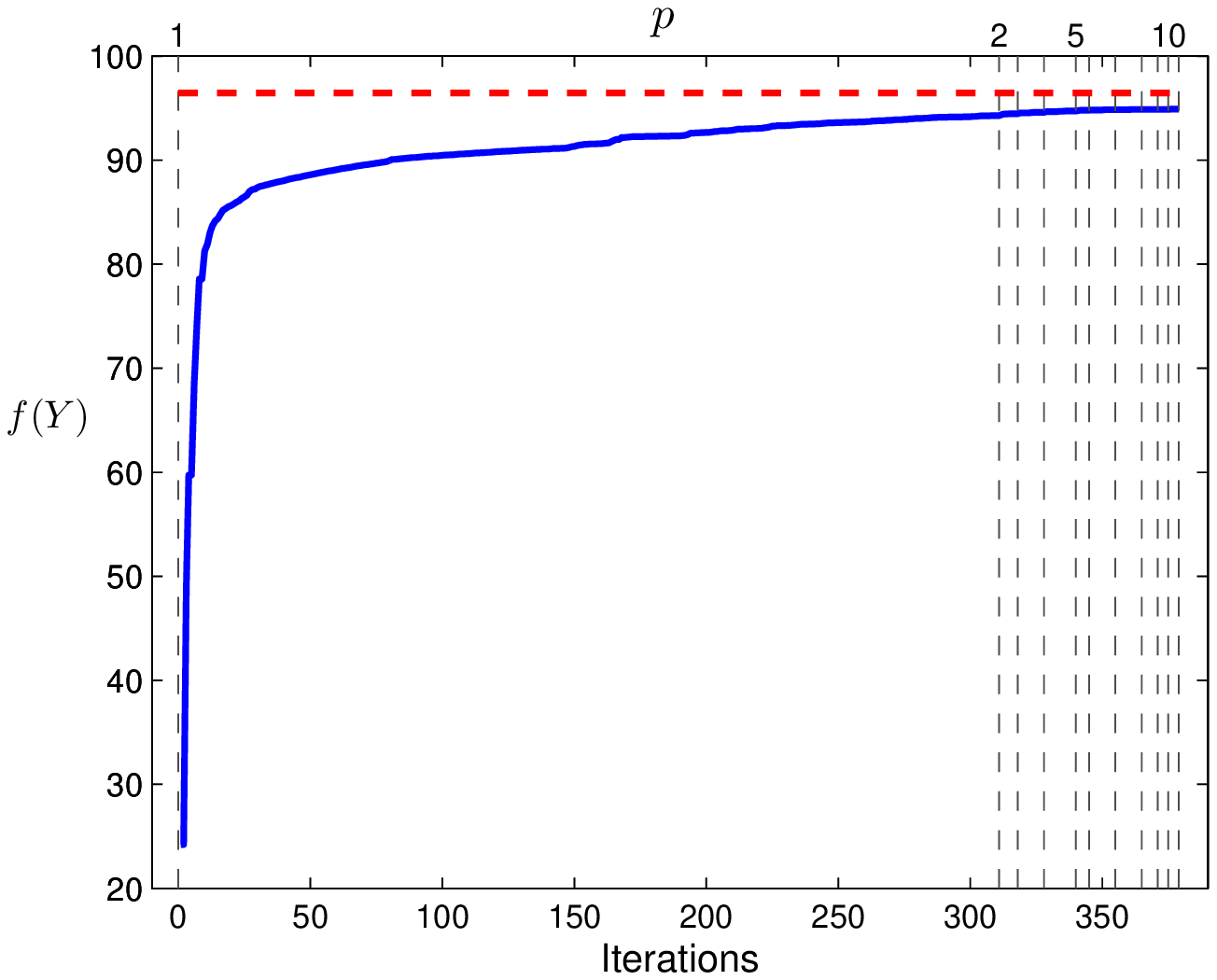}
\includegraphics[height=7cm,keepaspectratio]{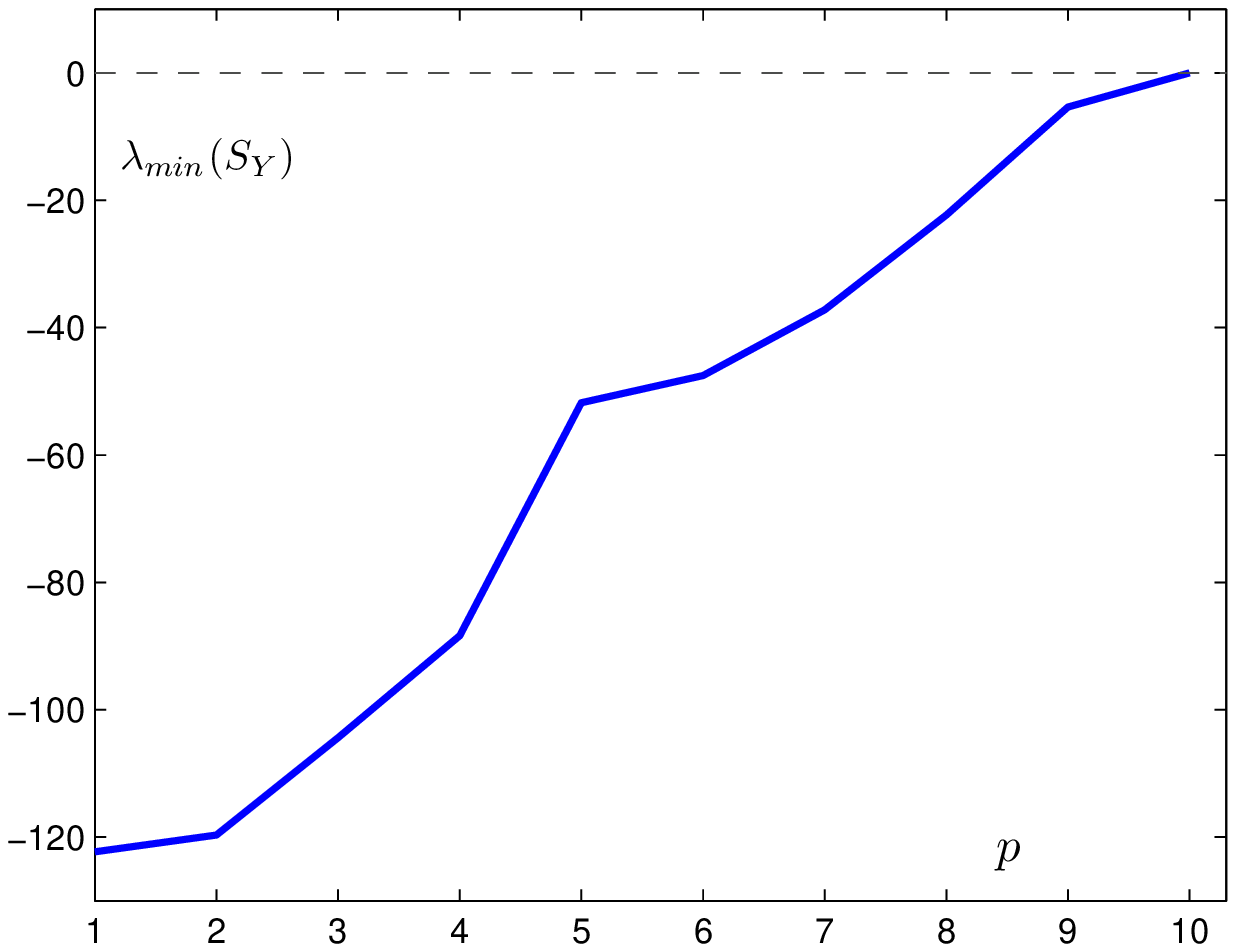}}
\caption{Left: monotone increase of $f(Y)=\mathrm{Tr}( Y^T \Sigma
Y) - \rho \sum_{i,j} h_{\kappa}((Y Y^T)_{ij})$ through the
iterations (bottom abscissa) and with the rank $p$ (top abscissa).
The dashed horizontal line represents the maximum of the nonsmooth
cost function in (\ref{eq:dspca2}). Right: evolution of the
smallest eigenvalue of $S_Y$.} \label{fig:DSPCA1}
\end{figure}

Figure \ref{fig:DSPCA2} provides some insight on the computational
time required by a Matlab implementation of Algorithm \ref{algo1}
that solves (\ref{eq:dspca3}). Square gaussian matrices $A$ have
been considered, i.e., $m=n$. On the left hand plot, Algorithm
\ref{algo1} is compared with the above mentioned heuristic and the
DSPCA algorithm. The right hand plot highlights the quadratic
complexity of Algorithm \ref{algo1} with the problem size $n$.

\begin{figure}[h]
\centerline{\includegraphics[height=7cm,keepaspectratio]{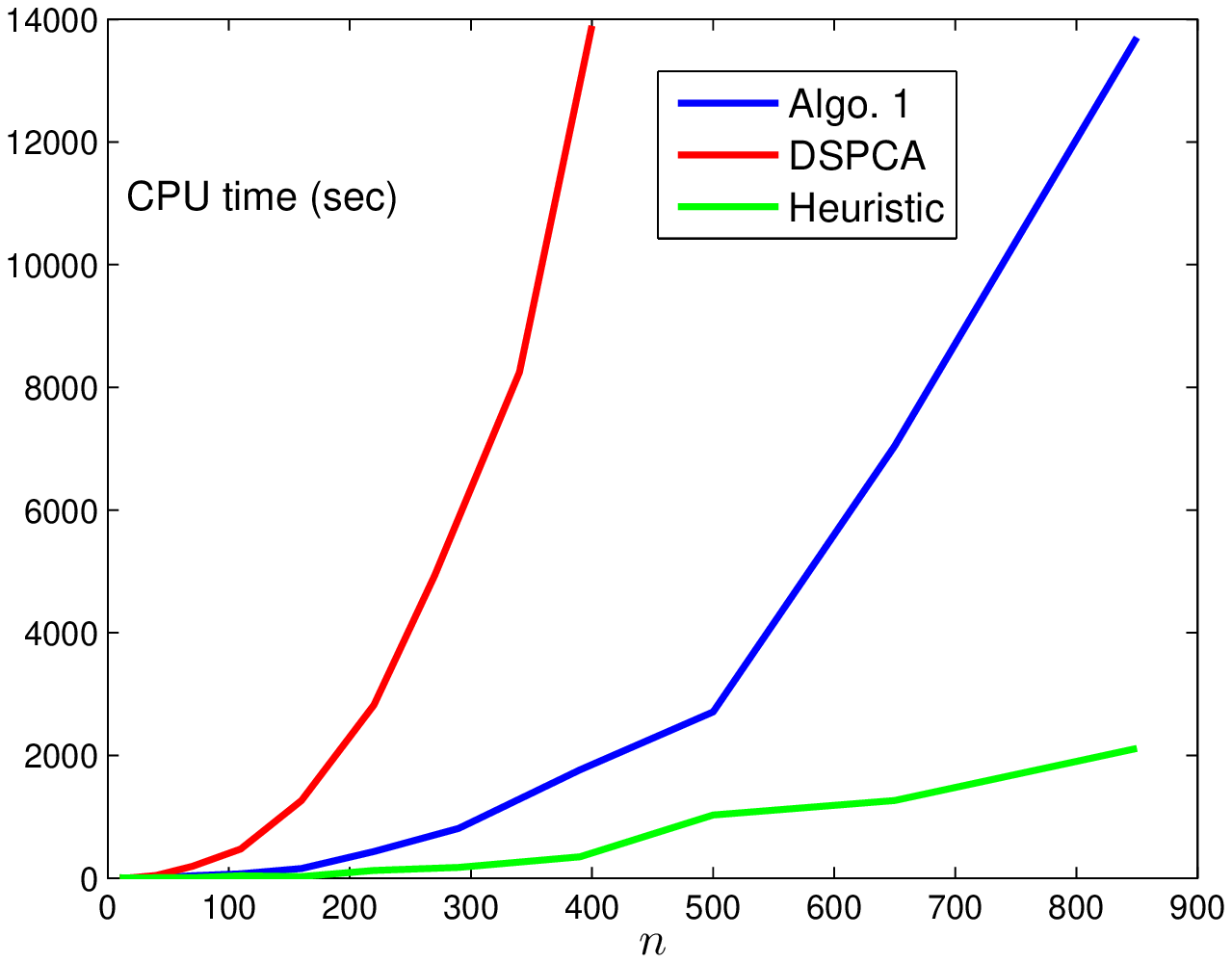}
\includegraphics[height=7cm,keepaspectratio]{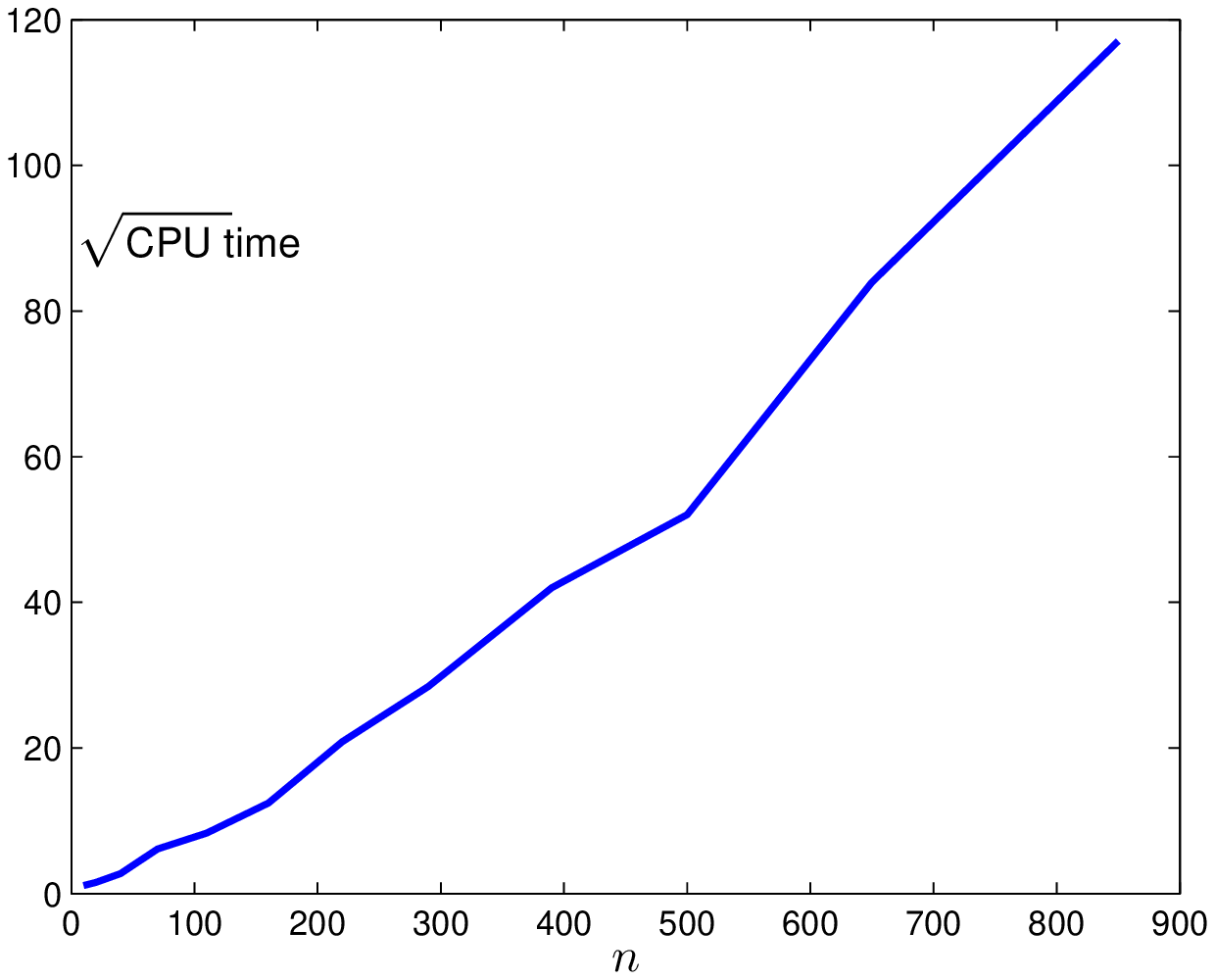}}
\caption{Right: Computational time for solving (\ref{eq:dspca3})
versus the problem size in the case $p=n$. Left: Square root of
the computational time versus $n$.} \label{fig:DSPCA2}
\end{figure}

\subsection{A second convex relaxation to the sparse PCA problem}
 Problem (\ref{eq:spca_formulation}) is shown in \cite{Aspremont07}
 to be strictly equivalent to
\begin{equation}
 \label{eq:spca1}
\begin{array}{ll}
\underset{\substack{z \in \mathbb{R}^{m}}}{\max}& \sum_{i=1}^n
((a_i^T
z)^2-\rho)_+, \\
\text{s.t.} & z^T z=1,
\end{array}
\end{equation}
where $a_i$ is the $i^{\mathrm{th}}$ column of $A$ and the
function $x_+$ corresponds to $\max(0,x)$. The auxiliary variable
$z$ enables to reconstruct the vector $x$: the component $x_i$ is
active if $(a_i^T z)^2 - \rho \geq 0$. As for the relaxation
previously derived in Section \ref{sec:DSPCA}, the vector $z$ is
lifted into a matrix $Z$ of the spectahedron,
\begin{equation}
\label{eq:spca2}
\begin{array}{cll}
&\underset{Z \in \mathbb{S}^{m}}{\max} & \sum_{i=1}^n \mathrm{Tr}(a_i^T Z a_i-\rho)_+ \\
 & \text{s. t.} & \mathrm{Tr}(Z)=1,\\
&  & Z \succeq 0,\\
\end{array}
\end{equation}\\
This program is equivalent to (\ref{eq:spca1}) in case of rank one
matrices $Z=z z^T$. Program (\ref{eq:spca2}) maximizes a convex
function and is thus nonconvex. The authors of \cite{Aspremont07}
have shown that, in case of rank one matrices $Z$, the convex cost
function in (\ref{eq:spca2}) equals the concave function
\begin{equation} \label{eq:ccv} f(Z)=\sum_{i=1}^n
\mathrm{Tr}(Z^{\frac{1}{2}}(a_i^T a_i-\rho
I)Z^{\frac{1}{2}})_+,\end{equation} where the function
$\mathrm{Tr}(X)_+$ stands for the sum of the positive eigenvalues
of $X$. This gives the following nonsmooth convex relaxation of
(\ref{eq:spca_formulation}),
\begin{equation}
\label{eq:spca3}
\begin{array}{cll}
&\underset{Z \in \mathbb{S}^{m}}{\max} & \sum_{i=1}^n \mathrm{Tr}(Z^{\frac{1}{2}}(a_i^T a_i-\rho I)Z^{\frac{1}{2}})_+ \\
 & \text{s. t.} & \mathrm{Tr}(Z)=1,\\
&  & Z \succeq 0,\\
\end{array}
\end{equation}\\
that is tight in case of rank-one solutions. This program is
solved via the factorization $Z=Y Y^T$ and optimization on the
quotient manifold $\mathcal{M}_{\mathcal{S}}$. In the case $Z=Y
Y^T$, function (\ref{eq:ccv}) equals
\[
f(Y)=\sum_{i=1}^n \mathrm{Tr}(Y^T (a_i^T a_i-\rho I) Y)_+,\] which
is a spectral function \cite{Aspremont07}. The evaluation of the
gradient and Hessian of $f(Y)$ are based on explicit formulae
derived in the papers \cite{Lewis96,Lewis01} to compute the first
and second derivatives of a spectral function. Since we are not
aware of any smoothing method that would preserve the convexity of
(\ref{eq:spca3}), Algorithm \ref{algo1} has been directly applied
in this nonsmooth context. In practice, no trouble has been
observed since all numerical simulations converge successfully to
the solution of (\ref{eq:spca3}). The computational complexity of
Algorithm \ref{algo1} for solving (\ref{eq:spca3}) is of order
$O(n m^2 p)$. The convex relaxation (\ref{eq:spca3}) of the sparse
PCA problem (\ref{eq:spca_formulation}) appears thus well suited
to treat large scale data with $m\ll n$, such as gene
expression data are.\\

Figure \ref{fig:Fig3} displays the convergence of Algorithm
\ref{algo1} for solving (\ref{eq:spca3}) with a random gaussian
matrix $A$ of size $m=100$ and $n=500$. The sparsity parameter
$\rho$ is chosen at 5 percent of the upper bound
$\bar{\rho}=\underset{i}{\max} \; a_i^T a_i$, that is derived in
\cite{Aspremont07}. The smallest eigenvalue $\lambda_{\min}$ of
the matrix $S_Y$ presents a monotone decrease once it gets
sufficiently close to zero.\\

\begin{figure}[h]
\centerline{\includegraphics[height=7cm,keepaspectratio]{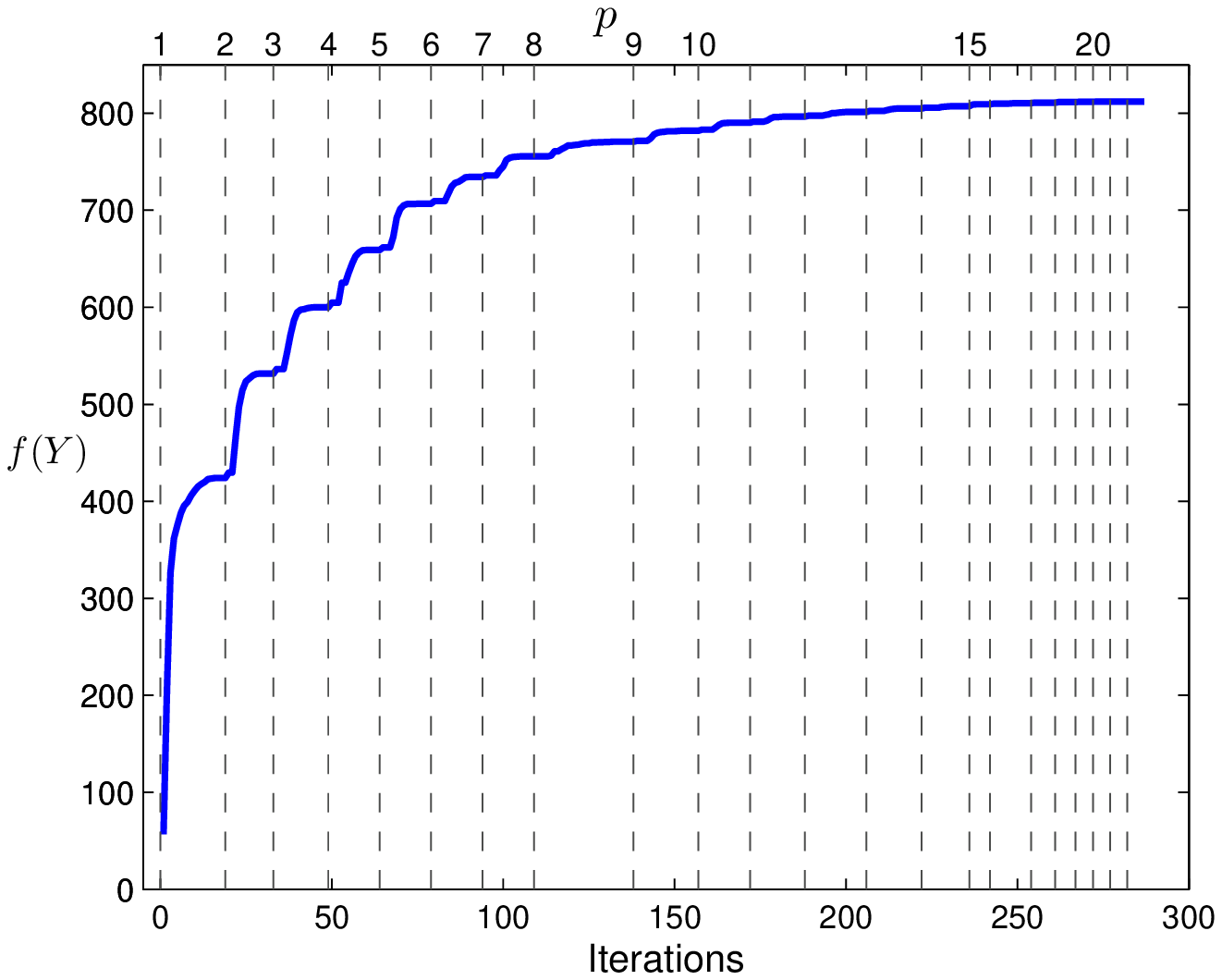}
\includegraphics[height=7cm,keepaspectratio]{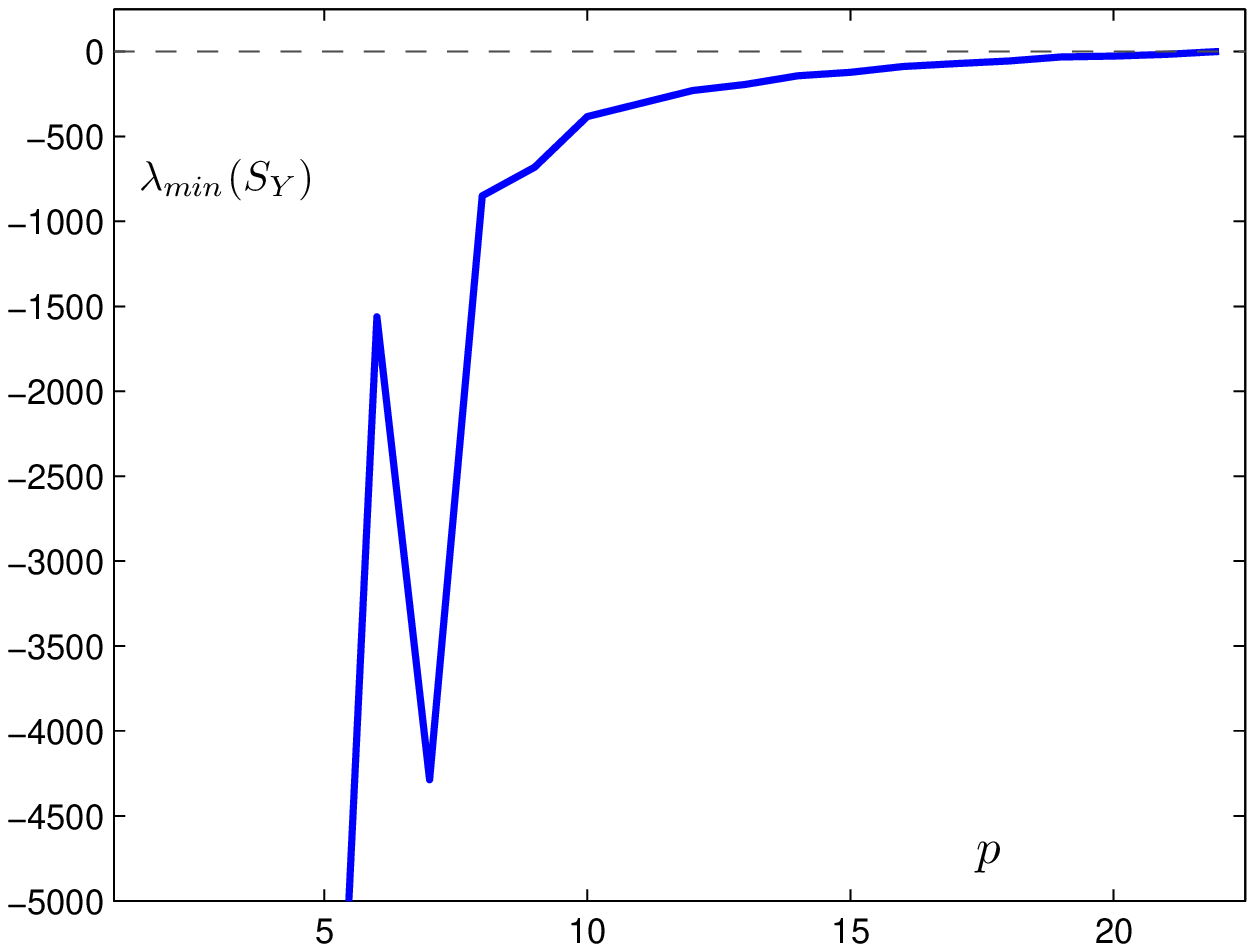}}
\caption{Left: monotone increase of the cost function through the
iterations (bottom abscissa) and with the rank $p$ (top abscissa).
Right: evolution of the smallest eigenvalue of $S_Y$.}
\label{fig:Fig3}
\end{figure}

Figure \ref{fig:SPCA_CPU} plots the CPU time required by a Matlab
implementation of Algorithm \ref{algo1} versus the dimension $n$
of the matrix $A$. The dimension $p$ has been fixed at 50 and $A$
is chosen according to a gaussian distribution. Figure
\ref{fig:SPCA_CPU} illustrates the linear complexity in $n$ of the
proposed sparse PCA method.

\begin{figure}[h]
\centerline{\includegraphics[height=7cm,keepaspectratio]{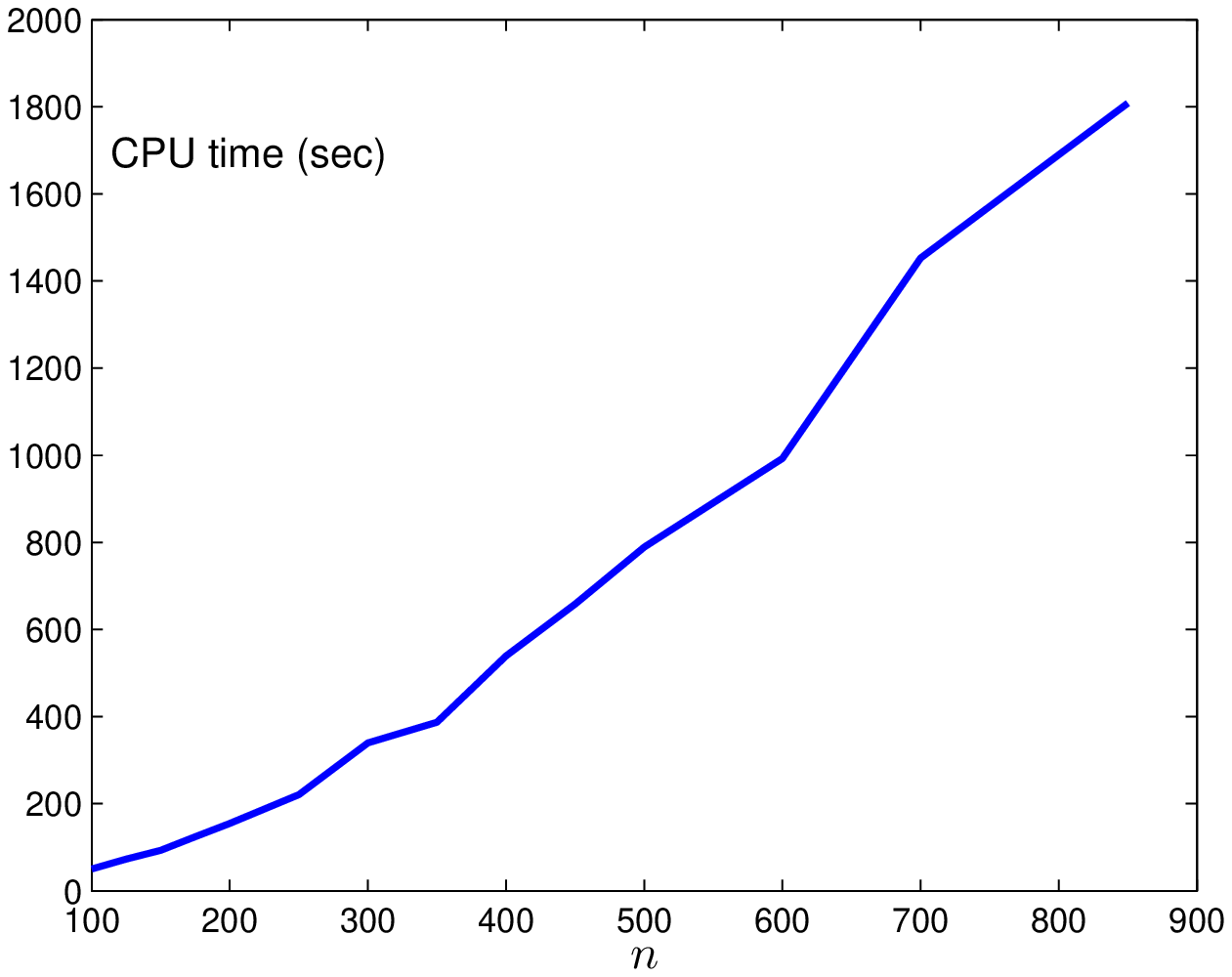}}
\caption{Computational time for solving (\ref{eq:spca3}) versus
the problem size $n$ in the case $p=50$.} \label{fig:SPCA_CPU}
\end{figure}

\subsection{Projection on rank one matrices}
Both convex relaxations (\ref{eq:dspca3}) and (\ref{eq:spca3}) are
derived from the reformulation of a problem defined on unit norm
vectors $x$ into a problem with matrices $X=x x^T$, which is an
equivalent formulation if $X$ belongs to the spectahedron and has
rank one. Within the derivation of both convex relaxations, the
rank one condition has been dropped. The solutions of
(\ref{eq:dspca3}) and (\ref{eq:spca3}) are therefore expected to
present a rank larger than one.\\

As previously mentioned, all numerical experiments performed with
the DSPCA algorithm \cite{Aspremont04}, which solves the nonsmooth
convex program (\ref{eq:dspca2}), led to a rank one solution.
Thus, the solution of the smooth convex relaxation
(\ref{eq:dspca2}) is expected to tend to a rank one matrix once
the smoothing parameter $\kappa$ gets sufficiently close to zero.
Figure \ref{fig:Randing1} illustrates this fact. It should be
mentioned that a matrix $X$ of the spectahedron has nonnegative
eigenvalues whose sum is one. Hence, $X$ is rank one if and only
if its largest eigenvalue equals one. In order to deal with
potential numerical problems in case of very small $\kappa$, we
sequentially solve problems of the type of (\ref{eq:dspca3}) with
a decreasing value of $\kappa$. The solution of each problem
initializes a new program (\ref{eq:dspca3}) with a reduced
$\kappa$.\\
\begin{figure}[h]
\centerline{\includegraphics[height=7cm,keepaspectratio]{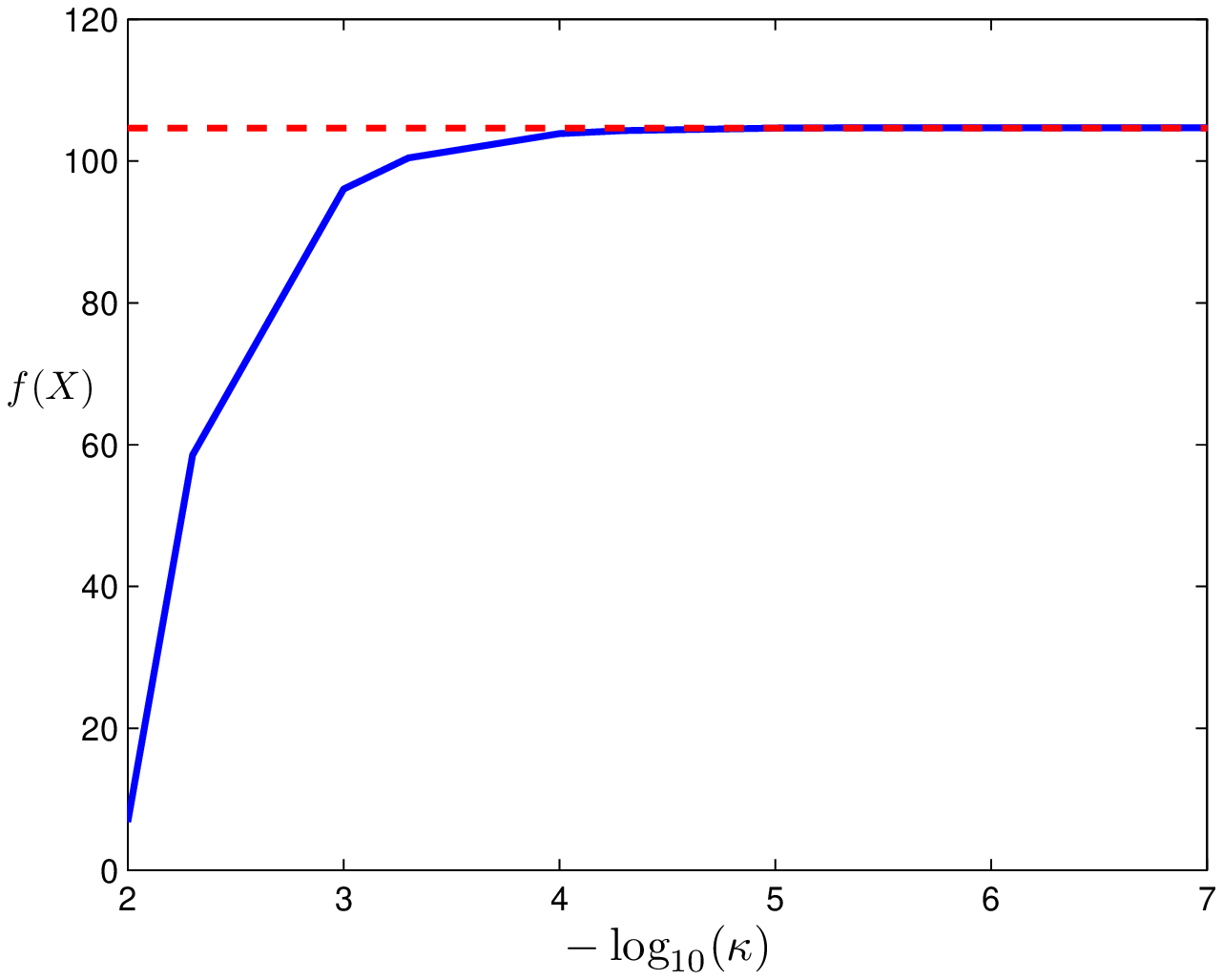}
\hspace{-1cm}
\includegraphics[height=7cm,keepaspectratio]{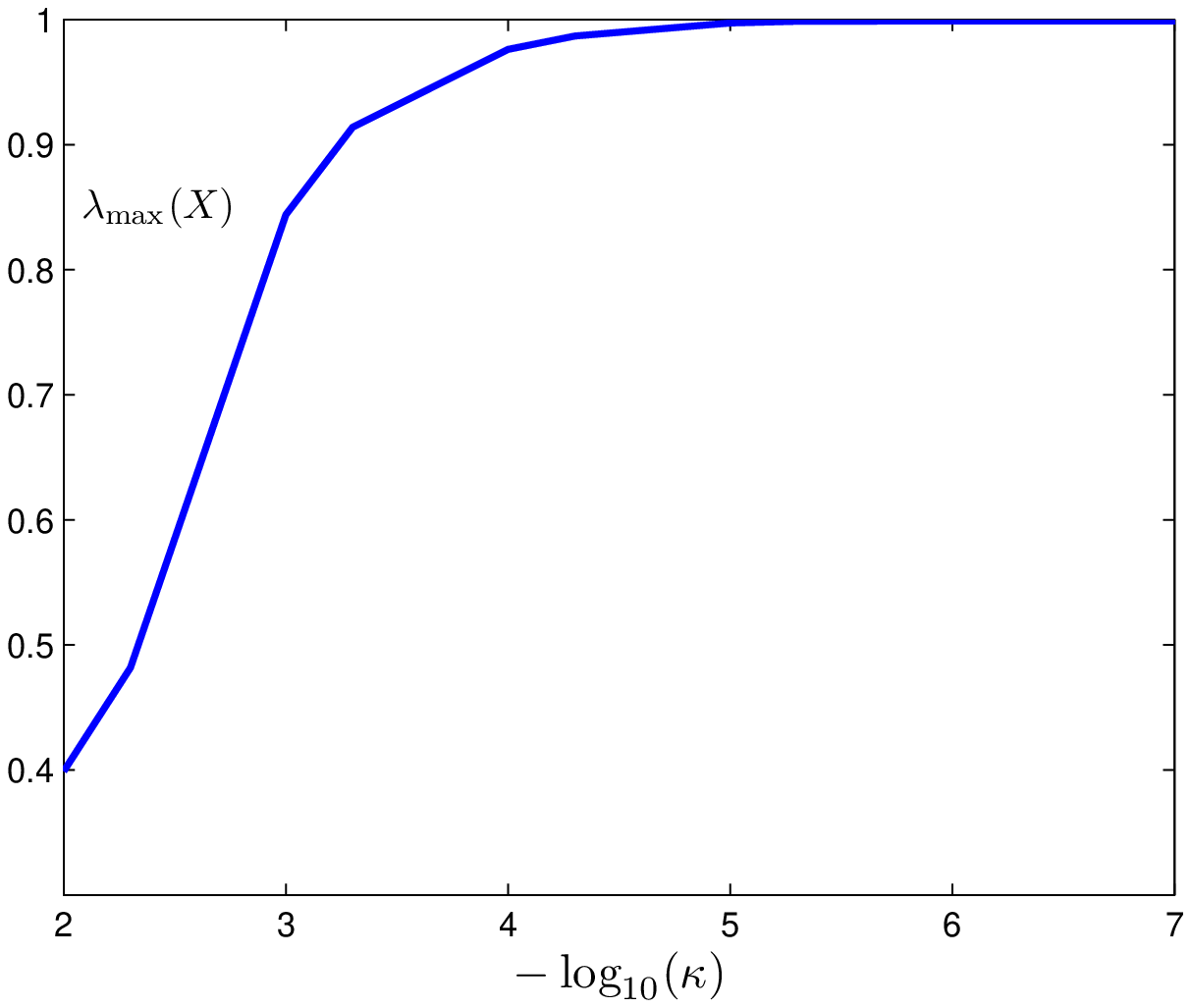}}
\caption{Left: evolution of the maximum the cost in
(\ref{eq:dspca3}) with the smoothing parameter $\kappa$. The
dashed horizontal line represents the maximum of the nonsmooth
cost function in (\ref{eq:dspca2}). Right: evolution of the
largest eigenvalue of the solution of (\ref{eq:dspca3}).}
\label{fig:Randing1}
\end{figure}

In contrast to (\ref{eq:dspca2}), the convex relaxation
(\ref{eq:spca3}) usually provides solutions with a rank that is
larger than one. The solution matrix $X$ has to be projected onto
the rank one matrices of the spectahedron in order to recover a
vector variable $x$. A convenient heuristic is to compute the
dominant eigenvector of the matrix $X$. A vector $x$ that achieves
a higher objective value in (\ref{eq:spca_formulation}) might,
however, be obtained with the following homotopy method. We
consider the program
\begin{equation}
\label{eq:randing}
\begin{array}{cll}
&\underset{Z \in \mathbb{S}^{m}}{\max} & \mu f_{cvx}(Z) + (1-\mu) f_{ccv}(Z) \\
 & \text{s. t.} & \mathrm{Tr}(Z)=1,\\
&  & Z \succeq 0,\\
\end{array}
\end{equation}\\
with the concave function,
\[  f_{ccv}(Z)=\sum_{i=1}^n \mathrm{Tr}(Z^{\frac{1}{2}}(a_i^T a_i-\rho I)Z^{\frac{1}{2}})_+\]
and the convex function,
\[ f_{cvx}(Z)=\sum_{i=1}^n \mathrm{Tr}(a_i^T Z a_i-\rho)_+,\]
and for the parameter $0 \leq \mu \leq 1$. As previously
mentioned, in case of rank one matrices $Z=zz^T$, the functions
$f_{ccv}(Z)$ and $f_{cvx}(Z)$ are identical and equal to the cost
function (\ref{eq:spca1}). For $\mu =0$, program
(\ref{eq:randing}) is the convex relaxation (\ref{eq:spca3}) and
the solution has typically a rank larger than one. If $\mu = 1$,
solutions of (\ref{eq:randing}) are extreme points of the
spectahedron, i.e., rank one matrices. Hence, by solving a
sequence of problems (\ref{eq:randing}) with the parameter $\mu$
that increases from zero to one, the solution of (\ref{eq:spca3})
is projected onto the rank one matrices of the spectahedron.
Program (\ref{eq:randing}) is no longer convex once $\mu > 0$. The
optimization method proposed in this paper then converges towards
a local maximizer of (\ref{eq:randing}).\\

Figure \ref{fig:Randing2} presents computational results obtained
on a random gaussian matrix $A \in \mathbb{R}^{150 \times 50}$.
This projection method is compared with the usual approach that
projects the symmetric positive semidefinite matrix $Z$ onto its
dominant eigenvector, i.e., $\tilde{Z} = z z^T$ where $z$ is the
unit-norm dominant eigenvector of $Z$. Let $f_{EVD}(Z)$ denotes
the function,\footnote{EVD stands for eigenvalue decomposition.}
\[f_{EVD}(Z)= f_{ccv}(\tilde{Z}) = f_{cvx}(\tilde{Z}).\]
Figure \ref{fig:Randing2} uses the maximum eigenvalue of a matrix
$Z$ of the spectahedron to monitor its rank. As previously
mentioned, any rank one matrix $Z$ of the spectahedron satisfies
$\lambda_{\max}(Z)=1$. The continuous plots of Figure
\ref{fig:Randing2} display the evolution of the functions
$f_{ccv}(Z)$ and $f_{EVD}(Z)$ during the resolution of the convex
program (\ref{eq:spca3}), i.e., $\mu = 0$ in (\ref{eq:randing}).
Point $A$ represents the solution obtained with Algorithm
\ref{algo1} by solving (\ref{eq:spca3}) at the rank $p=1$, whereas
$B$ and $B'$ stands for the exact solution of (\ref{eq:spca3}),
which is of rank larger than one. The dashed plots illustrate the
effect of the parameter $\mu$, that is linearly increased by steps
of 0.05 between the points $B$ and $C$ . For a sufficiently large
$\mu$, program (\ref{eq:randing}) presents a rank one solution,
which is displayed by the point C. One clearly notices that the
objective function of the original problem (\ref{eq:spca1}), which
equals $f_{EVD}(Z)$, is larger at $C$ than at $B'$. Hence, the
projection method based on (\ref{eq:randing}) outperforms the
projection based on the eigenvalue decomposition
of $Z$ in terms of achieved objective value.\\

\begin{figure}[h]
\centerline{\includegraphics[height=7cm,keepaspectratio]{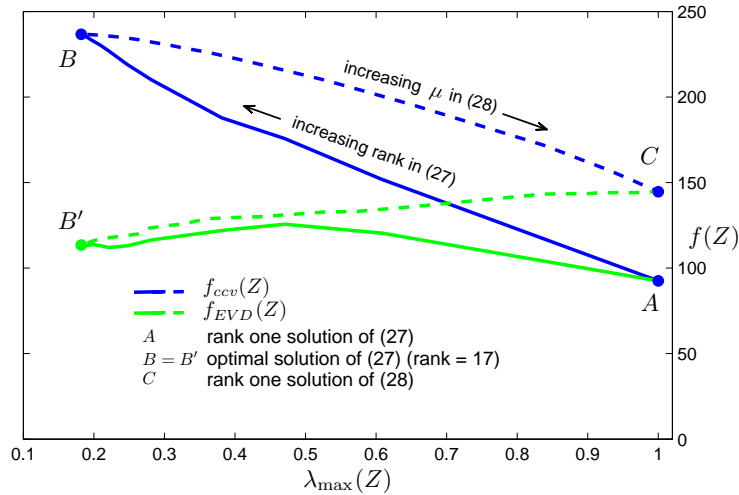}}
\caption{Evolution of the functions $f_{ccv}(Z)$ and $f_{EVD}(Z)$
in two situations. Continuous plots: resolution of the convex
program (\ref{eq:spca3}) ($\mu = 0$ in (\ref{eq:randing})). Dashed
plots: projection of the solution of (\ref{eq:spca3}) on a rank
one matrix by gradual increase of $\mu$.} \label{fig:Randing2}
\end{figure}

\section{Conclusion}
We have proposed an algorithm for solving a nonlinear convex
program that is defined in terms of a symmetric positive
semidefinite matrix and that is assumed to present a low-rank
solution. The proposed algorithm solves a sequence of nonconvex
programs of much lower dimension than the original convex one. It
presents a monotone convergence towards the sought solution, uses
superlinear second order optimization methods and provides a tool
to monitor the convergence, which enables to evaluate the quality
of approximate solutions for the original convex problem. The
efficiency of the approach has been illustrated on several
applications: the maximal cut of a graph and various problems in
the context of sparse principal component analysis. The proposed
algorithm can also deal with problems featuring a nonconvex cost
function. It then converges toward a local optimizer of the
problem.

\newcommand{\etalchar}[1]{$^{#1}$}
\providecommand{\bysame}{\leavevmode\hbox
to3em{\hrulefill}\thinspace}
\providecommand{\MR}{\relax\ifhmode\unskip\space\fi MR }
\providecommand{\MRhref}[2]{%
  \href{http://www.ams.org/mathscinet-getitem?mr=#1}{#2}
} \providecommand{\href}[2]{#2}

\end{document}